\documentclass[a4paper,10pt]{amsart}
\usepackage[utf8]{inputenc}
\usepackage{amsthm}
\usepackage{amsmath}
\usepackage{bm}
\usepackage{amsfonts}
\usepackage{soul}
\usepackage{amssymb}
\usepackage{mathtools}
\usepackage{appendix}
\usepackage{mathrsfs}
\usepackage{setspace}
\usepackage{xcolor}
\usepackage{todonotes}
\usepackage{thmtools} 
\usepackage[foot]{amsaddr}
\usepackage{comment}
\usepackage{bbm}

\usepackage[pdfdisplaydoctitle,colorlinks,breaklinks,urlcolor=blue,linkcolor=blue,citecolor=blue]{hyperref} 
\usepackage[nameinlink,capitalise]{cleveref}

\newcommand{\B}{\mathcal{B}}
\newcommand{\C}{\mathbf{C}}
\newcommand{\D}{\mathcal{D}}

\newcommand{\EE}{\mathbb{E}}
\newcommand{\F}{\mathcal{F}}

\renewcommand{\H}{\mathbf{H}}

\renewcommand{\L}{\mathbf{L}}
\newcommand{\N}{\mathbb{N}}

\newcommand{\PP}{\mathbb{P}}

\newcommand{\R}{\mathbb{R}}

\newcommand{\W}{\mathbf{W}}

\DeclareMathOperator{\e}{e}

\DeclareMathOperator{\grad}{\nabla}

\newcommand{\zlam}{\zeta_{\lambda}}
\newcommand{\tlam}{\theta_\lambda}
\newcommand{\BS}{\text{BS}}

\renewcommand{\epsilon}{\varepsilon}
\newcommand{\eps}{\varepsilon}

\newcommand{\derp}[2]{\frac{\partial #1}{\partial #2}}

\newcommand{\set}[1]{\left\{#1\right\}}

\newcommand{\pa}[1]{\big(#1\big)}
\newcommand{\bra}[1]{\left[#1\right]}

\newcommand{\norm}[2]{\left\|#1\right\|_{#2}}
\newcommand{\brak}[1]{\left\langle#1\right\rangle}

\newcommand{\expt}[2][]{\mathbb{E}_{#1}\left[#2\right]}

\newcommand{\weakstarto}{\overset{\star}{\rightharpoonup}}

\theoremstyle{plain}
\newtheorem{theorem}{Theorem}[section]

\newtheorem{corollary}[theorem]{Corollary}
\newtheorem{lemma}[theorem]{Lemma}
\newtheorem{proposition}[theorem]{Proposition}

\theoremstyle{definition}
\newtheorem{definition}[theorem]{Definition}

\theoremstyle{remark}
\newtheorem{remark}[theorem]{Remark}

\numberwithin{equation}{section}

\newenvironment{acknowledgements}{%
	\begin{abstract}
	}{%
	\end{abstract}
}

\newcommand{\op}{\grad^\bot(A+L)^{-1}}

\title[Invariant Measure for Stochastic MLQG Equations]{Existence of Invariant Measures for Stochastic Inviscid Multi-Layer Quasi-Geostrophic Equations}
\author[F. Butori]{Federico Butori}\address{Scuola Normale Superiore, Piazza dei Cavalieri 7, 56126, Pisa, Italy}
\email{\href{mailto:federico.butori at sns.it}{federico.butori at sns.it}}
\author[F. Grotto]{Francesco Grotto}
\address{Università di Pisa, Dipartimento di Matematica, 5 Largo Bruno Pontecorvo, 56127 Pisa, Italia}
\email{\href{mailto:francesco.grotto at unipi.it}{francesco.grotto at unipi.it}}
\author[E. Luongo]{Eliseo Luongo}\address{Scuola Normale Superiore, Piazza dei Cavalieri 7, 56126, Pisa, Italy}
\email{\href{mailto:eliseo.luongo at sns.it}{eliseo.luongo at sns.it}}
\author[L. Roveri]{Leonardo Roveri}
\address{Università di Pisa, Dipartimento di Matematica, 5 Largo Bruno Pontecorvo, 56127 Pisa, Italia}
\email{\href{mailto:leonardo.roveri at phd.unipi.it}{leonardo.roveri at phd.unipi.it}}
\date\today

\begin{document}
	
\maketitle
	
\begin{abstract}
	We consider an inviscid 3-layer quasi-geostrophic model with stochastic forcing in a 2D bounded domain. After establishing well-posedness of such system under natural regularity assumptions on the initial condition and the (additive) noise, we prove the existence of an invariant measure supported on bounded functions by means of the Krylov-Bogoliubov approach developed by Ferrario and Bessaih (Comm. Math. Phys. 377, 2020).
	\end{abstract}
	
	
\section{Introduction}\label{sec:introduction}
	
Multi-layer quasi-geostrophic equations are a standard model in the theoretical study of atmospheric and oceanic phenomena.
The equations are derived from general Navier-Stokes (shallow water) models \cite[Section 5.3]{Vallis2017}, and they describe the dynamics of stratified fluids characterized by an approximate balance between the pressure gradient and the Coriolis force.
	
Our study is focused on an inviscid and stochastically driven 3-layer quasi-geostrophic system:
\begin{gather}\label{eq:3lqg}
	\frac{\partial}{\partial t}q^i+(u^i\cdot \nabla)q^i=f^i, \quad
	u^i= \grad^\perp \psi^i, \quad   i=1,2,3,
\end{gather}
in which dynamics is expressed in terms of potential vorticity $q^i$ of the $i$-th layer and the associated stream function $\psi^i$, with $f^i$ being a random forcing term.
The interaction between different layers is encoded in the linear equations linking potential vorticity and stream functions,
\begin{equation}\label{eq:3lqgvallis}
\begin{cases}
    q^1=\Delta \psi^1+\beta y
	+\frac{c^2}{H_1 g_1}(\psi^2-\psi^1)+\frac{c^2}{g H_2} \psi^1,\\
	q^2=\Delta \psi2+\beta y
	+\frac{c^2}{H_2 g_1}(\psi^1-\psi^2)
	+\frac{c^2}{H_2 g_2}(\psi^3-\psi^2),\\
	q^3=\Delta \psi^3+\beta y
	+\frac{c^2}{H_3 g_2}(\psi^2-\psi^3)
	+\frac{c}{H_3} \eta,
\end{cases}
\end{equation}
the space domain is a bounded region $D\subset \R^2$ and the latter elliptic equations for potential vorticity are solved under Dirichlet boundary conditions.
In \eqref{eq:3lqgvallis}, parameters $H_i>0$ describe the basic-state thickness of single layers, $g_i=g\left(\rho_{i+1}-\right.$ $\left.\rho_i\right) / \rho_1$ the gravitational acceleration due to density differences between layers, 
$\eta$ is the height of the bottom topography, and finally $c$ is a variable Coriolis parameter.
There is no variation in the fluid densities $\rho_i$ (Boussinesq approximation).
Moreover, viscosity is absent and the fluid motion includes stochastic forcing modelling the interaction with small scales, coherently with the setup for the observation of geostrophic turbulence (cf. \cite[Chapter  9]{Vallis2017}).
The parameter $\beta>0$ rules the $\beta$-plane approximation of Coriolis force.
All the forthcoming arguments are easily generalized to an arbitrary number of interacting layers, and we restrict to the case of three layers for the sake of a clear exposition.
	
From a mathematical standpoint, when the $f^i$'s are deterministic or null the well-posedness theory of the PDE system described above is completely analogous to that of 2D Euler's equations, and there exists a unique solution for potential vorticity of class $L^\infty$. The proof follows by the classical argument of Yudovich \cite{YUDOVICH19631407}, and we refer to \cite{Chen2019} for details.
In order to preserve well-posedness, stochastic forcing terms must satisfy 
suitable regularity assumptions. Moreover, in order to observe stationary states of the system, friction terms must be included in the dynamics in order to compensate the input energy (see \eqref{eq:qgsto} below for a proper formulation of the stochastic PDE problem we consider). Both regularity of noise and the presence of friction are physically sensible (cf. \cite[p. 212]{Pedlosky2013}) and allow to consider long-time evolution of the system towards equilibrium states.
	
The aim of the present contribution is to rigorously establish an existence result for invariant measures, i.e. equilibrium states, supported by the well-posedness class $L^\infty$ of potential vorticity. Invariant measures and convergence to equilibrium in quasi-geostrophic and related models have been the object of numerous works in the viscous case: we mention for instance the recent works \cite{CarigiLuongo22,Carigi2023} and refer to \cite{Carigi2021} for an overview.
Equilibrium states for geophysical models in low regularity regimes were considered in \cite{GrottoPappalettera2021,GrottoPappalettera2022} (mirroring analogous techniques from the theory of Euler equations \cite{Albeverio1990,Flandoli2020,Grotto2020a,Grotto2020b,Grotto2022}), but to the best of our knowledge our result is the first to study invariant measures in inviscid geophysical model within the well-posedness regularity regime. Let us stress that in this context we can not rely on parabolic regularity theory, and instead need to resort to the averaging effect of stochastic forcing. 
	
We will establish well-posedness of a stochastic 3-layer quasi-geostrophic system by means of an inviscid limit approach, using non-physical viscous approximations that prevent vorticity creation at the boundary, thus avoiding a critical issue of viscous models in bounded domains (cf. \cite{Cottet88,GrottoLuongoMaurelli2022}).
Concerning invariant measures, our arguments will be based on the infinite-dimensional Krylov-Bogoliubov approach developed in \cite{Bessaih2020} for 2D Euler equations on a periodic domain, allowing to obtain existence of invariant measures supported by function spaces suited for the study of turbulent phenomena. In generalizing the techniques of \cite{Bessaih2020} to quasi-geostrophic systems, we are also able to improve their results in terms of Sobolev regularity both for well-posedness of the stochastic PDE system and for the support of invariant measures.
	
The paper is organized as follows. \Cref{sec:preliminaries} recalls basic notions to be used in the remainder of the paper, and with which we can rigorously present the main results of the paper and outline their proofs in \Cref{sec:overview}.
\Cref{sec:wellposedness,sec:stability} establish well-posedness of the stochastic PDE under consideration and properties of solutions, proving \cref{thm:main}
Finally, \Cref{sec:invariantmeasures} proves the main theorem of the paper, \cref{thm:invmeasure}.
	

 \section{Functional Analytic Setup}\label{sec:preliminaries}
	
Assume from now on that $D\subset \R^2$ is a simply connected, smooth bounded domain.
We denote $\D=C^\infty_c(D)$ and by $\D'=C^\infty_c(D)'$ its topological dual, i.e. the space of distributions. Brackets $\brak{\cdot,\cdot}$ will denote duality couplings in $\D\times \D'$, to be understood as $L^2$ products whenever both arguments belong to the latter.
	
We will also denote by $\brak{\cdot,\cdot}$ the $L^2$ scalar product for time-space functions, i.e. elements of spaces $L^p(0,T;L^q(D))$, whose norm will be denoted by $\norm{\cdot}{L^p_tL^q_x}$.
	

\subsection{Functional spaces}

We will rely on fractional Sobolev spaces and fractional powers of the Dirichlet Laplacian, which have various possible definitions, often non-equivalent.
We choose the approach base on the \emph{spectral} fractional Laplace operator. 
We refer to \cite{Caffarelli2016},\cite[Chapter 1]{LioMag}, \cite[Chapter 3]{Triebel83} for details on the material covered in this paragraph.
	
By $W^{s,p}(D)$, $s\in \N,\ p\in [1,+\infty]$, we denote Sobolev spaces obtained as the closure of $C^s(\bar{D})$ under the usual $L^p$-based Sobolev norm, and we set
\begin{equation*}
	H^s(D)=W^{s,2}(D),\quad W^{0,p}(D)=L^p(D).
\end{equation*}
Fractional Sobolev spaces $H^s(D)$, $0<s<1$, are the completion of $C^{1}(\bar{D})$) under the Sobolev-Slobodeckij norm
\begin{align*}
	\left( \lVert u\rVert^2_{L^2(D)}+\int_D \int_D \frac{(u(x)-u(y))^2}{|x-y|^{2+2s}}dxdy \right)^{1/2};
\end{align*}
the definition is extended to any $s = m+\sigma$, $m\in\N$, $\sigma \in (0,1)$, defining the norm
\begin{align*}
	\lVert u\rVert_{H^{s}(D)}^2=\lVert u\rVert_{H^m(D)}^2+\sum_{|\alpha| = m} \lVert D^\alpha u\rVert_{H^s(D)}^2.
\end{align*} 
Considering the closure of $\D$ instead of $C^{s}(\bar{D})$, we can define as above the spaces $W_0^{s,p}(D)$ and $H_0^s(D)$ (again, we set $W_0^{0,p}(D) = L^p(D)$). 
In particular, there exist an explicit relationship between these different spaces: for $0\leq s\leq 1/2$ it holds $H^s_0(D)=H^s(D)$, while for $s>1/2$ one has $H^s_0(D)\hookrightarrow H^s(D)$, the trace operator is well defined and $H^s_0(D)$ contains the null-trace elements. 
	
For each $s>0$ the topological dual of $H^{s}_0(D)$, denoted by $H^{-s}(D)$, is a subspace of the space $\D'$ of distributions on $D$, since $H^{s}_0(D)$ was obtained as the closure of $\D$ with respect to a suitable norm. 
Moreover, embeddings $H^s_0(D)\hookrightarrow L^2(D)\hookrightarrow H^{-s}(D)$ are continuous and dense. 
	
The following proposition gathers some classical statements on the relations between the spaces defined until now, intended in a two-dimensional setting.
	
\begin{proposition} The following hold:
\begin{itemize}
	\item (Sobolev embeddings) given $1< q<\infty$ and $\alpha> 0$, 
	\begin{equation*}
		H^{\alpha}(D)\hookrightarrow L^q(D), \qquad H_0^{\alpha}(D)\hookrightarrow L^q(D) 
	\end{equation*}
	whenever $\alpha-1\geq -\frac{2}{q}$. These embeddings are also compact if the inequality is strict;
	\item (Rellich Theorem)	given $\alpha>\beta\geq 0$,
	\begin{equation*}
		H^{\alpha}(D)\stackrel{c}{\hookrightarrow} H^{\beta}(D) , 
		\quad H_0^{\alpha}(D)\stackrel{c}{\hookrightarrow} H_0^{\beta}(D), 
		\quad H^{-\beta}(D)\stackrel{c}{\hookrightarrow} H^{-\alpha}(D);
	\end{equation*}
	\item (Morrey's inequality) given $l\in \N$ and $\lambda\in [0,1)$, for every $\alpha$ such that $\alpha-1>l+\lambda$ 
	\begin{equation*}
		H^{\alpha}(D)\stackrel{c}{\hookrightarrow}C^{l,\lambda}\left(\bar{D}\right),\qquad H_0^{\alpha}(D)\stackrel{c}{\hookrightarrow}C^{l,\lambda}\left(\bar{D}\right),
	\end{equation*} 
	and for every $p$ such that $1-\frac{2}{p}>l+\lambda$ 
	\begin{equation*}
		W^{1,p}(D)\stackrel{c}{\hookrightarrow}C^{l,\lambda}(\bar{D});
	\end{equation*}
	\item (Interpolation) given $\theta\in [0,1]$ and $(1-\theta)s+\theta s'\notin (\N+1/2)$, 
	\begin{align*}
		\left[H^s_0(D), H^{s'}_0(D)\right]_{\theta}&=H^{(1-\theta)s+\theta s'}_0(D),\\ \left[H^{-s}(D), H^{-s'}(D)\right]_{\theta}&=H^{-(1-\theta)s+\theta s'}(D) .
	\end{align*}
\end{itemize}
\end{proposition}
	
We are also going to use a different definition of fractional Sobolev space, based on Fourier series expansions.
There exists an orthonormal basis of $L^2(D)$ of eigenfunctions $\set{e_n}_{n\in\N}\subset H^1_0(D)$ of the Dirichlet Laplacian $-\Delta$, i.e. $-\Delta e_n=\lambda_n e_n$ for some $\lambda_n\in\R$.
For $s\geq 0$, spectral fractional Laplacian is defined as a densely defined, positive self-adjoint operator on $L^2(D)$ by
\begin{gather*}
	(-\Delta)^s u= \sum_{n=1}^\infty \lambda_n^{s} \hat u_n e_n\in L^2(D),\quad \hat{u}_n=\brak{ u,e_n} ,\quad u\in \mathsf{D}\pa{(-\Delta)^s},\\
	\nonumber
	\mathsf{D}\pa{(-\Delta)^s}=\set{u=\sum_{n=1}^\infty \hat u_n e_n: \hat u_n\in \R,
		\sum_{n=1}^\infty \lambda_n^{2s} |\hat u_n|^2<\infty}:= \mathcal{H}^s.
\end{gather*}
It is easy to check that $\mathcal{H}^s$, endowed with the norm
\begin{equation*}
	\lVert {u}\rVert_{\mathcal{H}^s}^2=\brak{(-\Delta)^{s/2}u,(-\Delta)^{s/2} u},
\end{equation*}
is a Hilbert space. Moreover, $\mathcal{H}^s=H^s(D)$ for $s\in (0,1/2)$, $\mathcal{H}^s=H^s_0(D)$ for $s\in (1/2,3/2)$, and $\mathcal{H}^s=H^1_0(D)\cap H^s(D)$ for $s\in [3/2,2]$ in the sense that they coincide as function spaces and their norms are equivalent (cf. \cite{Caffarelli2016}).
For completeness, in the case $s=1/2$, $\mathcal{H}^{1/2}$ coincides with the so-called Lions-Magenes space $H^{1/2}_{00}(D)$ \cite{LioMag}, but we are not going to use it. 
	

\subsection{Weak$\star$ and bounded weak$\star$ topologies}

In the last part of this work we will make use of some notions on weak topologies that we now recall. Consider the space $L^\infty$ with the norm topology $\tau_n$. The weak$\star$ topology $\tau_\star$ on $L^\infty$ is the coarsest topology among those making the duality pairing $\brak{\cdot, \psi}:L^\infty \rightarrow \R$ continuous, for all $\psi\in L^1$.
	
A sequence of elements $\phi_n\in L^\infty$ is said to converge to some $\phi\in L^\infty$ in the weak$\star$ sense, $\phi_n\overset{\star}{\rightharpoonup} \phi$, if $\brak{\phi_n, \psi} \rightarrow \brak{\phi, \psi}$ for all $\psi\in L^1$. 
Due to the impossibility of metrizing $\tau_\star$, such notion of convergence does not completely identify the topology. 
However, we can circumvent this issue by only considering bounded sets as, being $L^1$ separable, finite-radius balls are weakly$\star$ metrizable (cf. \cite{brezis} for an extended discussion). 
Since we will mostly be working with sequences, it is then convenient to introduce the notion of bounded weak$\star$ topology $\tau^b_{\star}$, that is, the finest topology on $L^\infty$ coinciding with the weak$\star$ topology on all norm-bounded sets. 
We have the strict inclusions
\begin{equation*}
	\tau_\star \varsubsetneq \tau^b_\star \varsubsetneq \tau_n . 
\end{equation*}
As a consequence of the weak$\star$ metrizability of balls, a map $F:L^\infty \rightarrow \R$ is sequentially $\tau_\star$-continuous (i.e. $F(\phi_n)\rightarrow F(\phi)$ whenever $\phi_n\overset{\star}{\rightharpoonup} \phi$) if and only if $F$ is $\tau^b_\star$-continuous (see \cite{Bessaih2020}). 
	
Concerning measurability, Borel sets generated by $\tau_\star$ and $\tau_\star^b$ coincide, but they are a proper subset of those generated by $\tau_n$,
as opposed to the case where the space is separable in the strong topology, 
in which $\tau_\star^b$, $\tau_\star$ and $\tau_n$ coincide, cf. \cite{talagrand}. 
We will therefore always specify with respect to which norm measurability is intended.
	

\section{Overview of Main Results and Methodology} \label{sec:overview}

 
\subsection{Notation for the multi-layer system}

The quasi-geostrophic system \eqref{eq:3lqg}, \eqref{eq:3lqgvallis}
can be written as
\begin{equation}\label{eq:detsetting}
	\left( \frac{\partial}{\partial t} + u^i \cdot \grad \right) \bigg( \zeta^i + \sum_{j=1}^3 \tilde{l}_{ij} \psi^j \bigg) = \tilde f^i,
	\quad i=1,2,3,
\end{equation}
where $u^i$ is the velocity field of the $i$-th layer, $\zeta^i = \grad \times  u^i = \partial_2u^i_1 - \partial_1u^i_2$ is a scalar function and $\psi^i$ is the (scalar) stream function associated to $u^i$, i.e. $\psi^i$ satisfies $\grad^\perp \psi^i = u^i$, and $\Delta \psi^i = \zeta^i$. The vector $\tilde{f}^i$ accounts for external forces, while by $\tilde L = (\tilde l_{ij})_{i,j=1,2,3}$ we denote the layers' interaction matrix:
in the case of \eqref{eq:3lqgvallis} it has the form 
\begin{equation*}
	\tilde{L} = \begin{pmatrix}
	-\lambda_1 & \lambda_1 & 0 \\
	\lambda_2 & - 2 \lambda_2 & \lambda_2 \\
	0 & \lambda_3 & - \lambda_3
    \end{pmatrix} ,
\end{equation*}
where $\lambda_i$'s are combinations of physical constants of the problem. 
	
Setting $u = (u^1, u^2, u^3)$, $\psi = (\psi^1, \psi^2, \psi^3)$, $f = (f^1, f^2, f^3)$ and denoting by $\tilde{q}$ the potential vorticity
\begin{equation*}
	\tilde q = \Delta \psi + \tilde L \psi
\end{equation*}
(where $\Delta$ stands for $\operatorname{diag}(\Delta, \Delta, \Delta)$ abusing notation), 
the system \eqref{eq:detsetting} can be conveniently rewritten in the form 
\begin{equation}\label{eq:qgdet}
	\derp{}{t} \tilde q + u \cdot \grad \tilde q = \tilde f,
\end{equation}
in which the velocity field $u$ can be formally reconstructed from $\tilde{q}$ via $u = \grad^\perp (\Delta + \tilde L)^{-1} \tilde{q}$. 	
In the following, it may be convenient to assume that the matrix $\tilde L$ is symmetric and negative definite. 
This is not consistent with the physical reality that the model aims at describing, but we can circumvent the problem by introducing a scaling factor in the equation. 
Indeed, if $ D = \operatorname{diag}(h_1, h_2, h_3)$ such that $h_i>0$ and $h_i \lambda_i = \lambda$ for all $i=1,2,3$, then $D\tilde L$ is symmetric and negative semidefinite (since for example $(1,1,1)^\top \in \ker(L)$). 
By multiplying \eqref{eq:qgdet} by $D$, we can thus define 
\begin{equation*}
	q = (D\Delta + D\tilde L) \psi , \quad u = \grad^\perp (D\Delta + D\tilde L)^{-1} \tilde q = \tilde u , \quad f = D\tilde f,
\end{equation*}
and obtain an equation of the same form, 
\begin{equation*}
	\derp{}{t}q + u \cdot \grad q = f,
\end{equation*}
where now the operator $(D\Delta + D\tilde L)$ is symmetric and negative definite. 
For the sake of notation, in the following we write it as $(A+L)$, with  
\begin{equation*}
	A \coloneqq D\Delta \quad \text{and} \quad L \coloneqq D \tilde{L} ,
\end{equation*}
now assuming that $L$ is symmetric and negative semidefinite. 
	
In order to properly define the PDE problem and to give meaning to the inverse of $(A+L)$ we need to specify boundary conditions, and we will choose Dirichlet conditions,
\begin{equation*}
	\psi(t, \cdot)_{|\partial D}= 0 .
\end{equation*}
	
As for function spaces we will be working with, dealing with a multi-layer-problem we are interested in vector fields having some Sobolev regularity. 
Let us therefore define, for $p\in [1,+\infty]$, $s\in \N$,
\begin{align*}
	\W^{s,p}=\{u=(u^1,u^2,u^3)^t:\ u^1,u^2,u^3\in W^{s,p}\} , 
	\quad 
	\norm{u}{\W^{s,p}} = \sum_{i=1}^3 \norm{u^i}{W^{s,p}}.
\end{align*}
and for $s\in\R$ 
\begin{align*}
	\H^s=\{u=(u^1,u^2,u^3)^t:\ u^1,u^2,u^3\in H^s\},\quad \norm{u}{\H^s}^2=\sum_{i=1}^3 \norm{u^i}{H^s}^2.
\end{align*}
In an analogous way we can define $\W_0^{s,p}$ and $\H_0^{s}$ (for $s\geq 0$). 
When $s=0$, we set $\W^{s,p}_0(D)=\W^{s,p}(D)=\mathbf{L}^p(D)$.
	
	
\subsection{Solutions of Quasi-Geostrophic Equations}
	
We will make use of the following classic result for elliptic boundary value problems, for which we refer to \cite{AmbrosioPDe}.
\begin{lemma}\label{elliptic estimate}
For every $q\in \L^{\infty}(D)$ the boundary value problem 
\begin{equation}\label{QGBVP}
\begin{cases}
	q = A \psi + L \psi,  & x\in D, \\
	\psi(x) = 0, & x \in \partial D,
\end{cases} 
\end{equation}
has a unique weak solution $\psi\in \W^{2,p}(D)$. 
Moreover, for every $i=1,\ldots, 3$,
\begin{equation*}
	\|\grad^2 \psi^i\|_{\L^{p}}\le Cp\|q^i\|_{\L^{\infty}} \qquad \forall p\in [1, \infty).
\end{equation*}
\end{lemma}
We are interested in a stochastic version of \eqref{eq:qgdet}. Consider a filtered probability space $(\Omega, \F,\F_t, \PP)$, let $\{\rho_k\}_{k \geq 0}$ be an orthonormal basis of $\L^2$ made of eigenvectors of the operator $(A+L)$ and define the Wiener process  
\begin{equation*}
	W_t(x) \coloneqq \sum_{k=0}^{+\infty} c_k \rho_k(x) W^k_t ,
\end{equation*}
where $\{c_k\}_k \subset \R$ and $\{W^k\}_k$ are independent standard $\R^3$-valued Brownian motions. Assume that $\{\F_t\}_{t\ge 0}$ is the filtration generated by this family of Brownian motions. 
This choice implies that $W_t|_{\partial D} \equiv 0$.
Notice also that one can consider $t\in\R$ just by taking $W^k_t \coloneqq \tilde{W}^k_{-t}$ for $t < 0$ where $\{\tilde{W}^k\}_k$ are independent copies of  $\{W^k\}_k$.	
The space regularity for $W$ is encoded in its Fourier coefficients.
From now on we will assume that $W_t\in\H^{5/2}$ for all $t\in \R$, hence 
\begin{equation*}
	\sum_{k}c_k^2\|\rho_k\|^2_{\H^{5/2}}< \infty .
\end{equation*}
	
The problem under consideration can now be formulated as 
\begin{equation}\label{eq:qgsto} 
\begin{cases}
	d q_t + \left[ (u_t \cdot \grad) q_t + \gamma q_t \right] d t = d W_t, \qquad &\text{on } D, \\
	\grad \cdot u_t = 0, \,\, u_t = \grad^\bot (A+L)^{-1} q_t, &\text{on } D, \\
	q_t = 0, & \text{on } \partial D,
\end{cases}
\end{equation}
where $\gamma>0$ is a damping coefficient ruling the friction term that prevents the accumulation of energy introduced by the random forcing
(we refer again to \cite{Pedlosky2013} for a physical motivation).
Recall that, given a filtered probability space $(\Omega, \F, \{\F_t\}_{t\ge 0}, \PP)$, a collection of random variables $X_t(\omega), \ t\in [0, T], \ \omega \in \Omega$ taking values in a Banach space $B$ is an adapted stochastic process if for every $t\in \R$ $X_t$ if $\F_t$-measurable. It is said \textit{progressively measurable} if for every $t\le T$, the mapping $[0, t]\times \Omega \ni (s, \omega)\rightarrow X_s(\omega)\in B$ is measurable with respect to $\B([0, t])\otimes \F_t$. 

\begin{definition}\label{def:solutionqgsto}
Given $q_0\in \L^2$, a stochastic process $q_t$ such that 
\begin{equation*}
	q\in C_w(0, T; \L^2) \qquad \PP\text{-a.s.}
\end{equation*}
is said to be a weak solution of \eqref{eq:qgsto} if it is adapted with respect to $\{\F_t\}_{t\leq0}$ and for all $\phi \in \H^1_0\cap \H^2$ and for all $t\in[0,T]$, it holds $\PP$-almost surely 
\begin{equation}\label{eq:sol1}
	\brak{q_t, \phi} 
	= 
	\brak{q_0, \phi} 
	+ \int_0^t \brak{q_s, u_s \cdot \grad \phi }ds  
	- \gamma\int_0^t\brak{q_s, \phi}ds 
	+ \sum_{k = 0}^\infty c_k\brak{\rho_k, \phi}W^k_t.
\end{equation}
\end{definition}
The regularity requirement on test functions is necessary to make sense of the nonlinear term. We will be able to loosen it when assuming better integrability of the initial datum. The following is the first main result of the present paper:
	
\begin{theorem}[Well-posedness]\label{thm:main}
For every $q_0\in \L^{\infty}$ there exists a pathwise unique solution to equation \eqref{eq:qgsto} in the sense of \cref{def:solutionqgsto}. 
Moreover, $q\in C_w(0, T; \L^\infty)$ $\PP$-almost surely, $q_t$ is progressively measurable as a process with values in $(\L^\infty, \tau_\star)$ and the formulation \eqref{eq:sol1} is valid for all $\phi \in \H^1_0$. 
It also holds that the difference $q-W\in W^{1,2}(0, T; \H^{-1})$ $\PP$-almost surely. 
		
In addition, if $q_0\in \W^{1,4}$, then $q\in C_w(0,T; \W^{1,4})$ $\PP$-almost surely and $q_t$ is progressively measurable as a process in $\W^{1,4}$.
\end{theorem}
	
Existence of solutions will be established by means of a vanishing viscosity argument in  \cref{prop:existence}, while uniqueness follows from the standard argument of Yudovich \cite{YUDOVICH19631407}, cf. \cref{prop:uniqueness}.

The following extends the notion of weak solution to time-dependent test functions,
and it will prove to be a convenient formulation in the forthcoming arguments.
\Cref{def:solutionqgsto,def:sol2} are in fact equivalent, as one can prove along the lines of \cite[Lemma 3]{luongo21inviscid}.

\begin{definition}\label{def:sol2}
Given $q_0\in \L^2$, a stochastic process $q^\eps_t$ such that 
\begin{equation*}
	q\in C_w(0, T; \L^2) \qquad \PP\text{-a.s.}
\end{equation*}
is said to be a weak solution of \eqref{eq:qgsto} if it is adapted with respect to $\{\F_t\}_{t\leq0}$  and, for all $\phi \in C(0,T; \H^1_0\cap \H^2)\cap C^1(0,T; \L^2)$ and for all $t\in[0,T]$, $\PP$-almost surely 
\begin{multline}\label{eq:sol2}
    \brak{q_t, \phi_t} 
	= \brak{q_0, \phi_0} + \int_0^t \brak{q_s, \partial_s \phi_s} d s + \int_0^t \brak{q_s, u_s \cdot \grad \phi_s }ds  - \gamma \int_0^t \brak{q_s, \phi_s} d s \\ 
	+ \sum_{k = 0}^\infty c_k \brak{\rho_k, \phi_t} W_t^k - \sum_{k = 0}^\infty \int_0^t c_k \brak{\rho_k, \partial_s \phi_s} W_s^k d s
\end{multline}
\end{definition}
	

\subsection{Properties of Solutions and Existence of Invariant Measures}\label{sec:prop of solutions}
	
Consider the solution $q_t$ of \eqref{eq:qgsto} obtained in \cref{thm:main}.
Let $B_b(\L^\infty, \tau_\star)$ be the set of bounded real-valued functions on $\L^\infty$ that are measurable with respect to the weak$\star$ topology.
For $\phi\in B_b(\L^\infty, \tau_\star)$ we denote 
\begin{equation*}
	P_t\phi(\chi)\coloneqq\expt[]{\phi(q(t, \chi))}.  
\end{equation*}
If $q_t$ possesses the Markov property, the latter is the associated Markov semigroup of operators.
This is in fact the case, but the functional setting in $\L^\infty$ requires a careful argument, to which we devote \Cref{ssec:markov}.
In particular, \cref{lemma:markov} provides sufficient conditions for the Markov property of stochastic flows taking values in topological vector spaces,
under the assumption that a separable Banach space embedded in the latter is preserved by the evolution.
In order to show that such condition is satisfied in our particular case, that is for the solution of \cref{thm:main},
we can rely on the stability of solutions in $\W^{1,4}$, to be proved in \cref{sec:stability}.
	
Once established the Markov property of the stochastic process under consideration,
we can adopt the following standard definition of invariant measure:
\begin{definition}
An invariant measure for the system \eqref{eq:qgsto} is a probability measure $\mu$ on $(\L^{\infty}(D), \B(\tau^b_\star))$ such that
\[
    \int P_t f d\mu =\int fd\mu, \qquad \forall t\ge 0, \forall \phi \in C_b(L^\infty, \B(\tau^b_\star)).
\]
\end{definition}
\Cref{sec:invariantmeasures} will be devoted to the proof of the anticipated existence theorem for invariant measures:
\begin{theorem}\label{thm:invmeasure}
Under the hypothesis of \cref{thm:main}, there exist at least one invariant measure for the problem \eqref{eq:qgsto}.
\end{theorem}
	
The argument consists in averaging over larger and larger time intervals the law of the stochastic process $q_t$,
in the classical Krylov-Bogoliubov fashion. The crucial step is thus to prove tightness of the time-averaged law of $q_t$:
as in \cite{Bessaih2020} we rely on an idea introduced by Flandoli \cite{Flandoli1994} in order to obtain the necessary uniform bounds,
to be detailed in \cref{prop:flandolitrick}.
	
\begin{remark}\label{rmk:regularity}
Our regularity assumption on the Wiener process $W_t$ slightly improves the ones in \cite{Bessaih1999,Bessaih2020}, where the authors assumed regularity in $H^{3+}$ for the stochastic forcing (at the level of vorticity). In particular, less regularity on the noise is required both to perform the vanishing viscosity argument, thanks to the introduction of an additional stochastic term which vanishes in the limit, and to ensure the stability of the solution in $W^{1,4}$ thanks to a more careful energy estimate. As a consequence, either \autoref{thm:main} and \autoref{thm:invmeasure} are improved, as they hold under weaker assumption on the space regularity of the noise.
\end{remark}


\section{Well-posedness of Stochastic Quasi-Geostrophic Systems}\label{sec:wellposedness}
	
In order to study the existence of solutions to \eqref{eq:qgsto} it is convenient to focus on the quantity $\eta = q - W$, and to study a vanishing viscosity limit for the dynamics of $\eta$.
Consider the Navier-Stokes system
\begin{equation}\label{NSE}
\begin{cases}
	d q_t^\epsilon + ( u_t^\epsilon \cdot \grad q_t^\epsilon + \gamma q_t^\epsilon ) d t = \epsilon^2 \Delta q_t^\epsilon d t + -\eps^2\Delta W_tdt + d W_t, \\
	u^\epsilon = \grad^\bot (A+L)^{-1} q^\epsilon, \\
	q^{\eps}_t|_{\partial D}= 0,
\end{cases}
\end{equation}
including both a viscosity term $\eps^2\Delta q_tdt$ and a further stochastic term $-\eps^2\Delta W_tdt$ that will be instrumental in dealing with the quantity $\eta = q - W$ (see \cref{rmk:aboutlemma:ViscSolBound} below). Notice that by assumption such additional stochastic forcing takes values in $C(0,T; \H^{1/2})$.
	
The condition $u^\epsilon\cdot n= 0 $ on $\partial D$ is implicit in the very definition of $u^\epsilon$, as reconstructed from the operator $\op$ with zero boundary condition. 
The notion of solution for \eqref{NSE} is similar to the one already given for the problem with $\eps=0$, except for the fact that a second order term in the equation allows to consider more regular solutions.
	
\begin{definition}\label{def:solutionNSE}
Given $q_0\in \L^2$, a stochastic process $q^\eps_t$ such that 
\begin{equation*}
	q^\eps(\omega)\in C(0, T; \L^2)\cap L^2(0, T; \H^1_0) \qquad a.s.
\end{equation*}
and is progressively measurable with respect to the norm-induced topologies is said to be a weak solution of \eqref{NSE} if, for all $\phi \in \H^1_0\cap \H^2$ and every $t\in[0,T]$, it holds $\PP$-almost surely
\begin{multline*}\label{eq:sol1-NSE}
	\brak{q^\eps_t, \phi}  
	= \brak{q^\eps_0, \phi} + \int_0^t \brak{q^\eps_s, u^\eps_s \cdot \grad \phi }ds  - \gamma\int_0^t\brak{q^\eps_s, \phi}ds \\ 
	+ \eps^2\int_0^t{\brak{\grad W_s - \grad q^\eps_s, \grad\phi_s}}ds + \sum_{k\ge 0}c_k\brak{\rho_k, \phi}W^k_t ,
\end{multline*}
where $u^\epsilon_t\coloneqq \grad^\perp (A+L)^{-1}q^\eps_t$.
\end{definition}
\noindent
As in the case of \cref{def:solutionqgsto,def:sol2}, we can use an equivalent time-dependent definition for \eqref{NSE}:
\begin{lemma}
Let $q^\eps$ be solution of the system \eqref{NSE}. Then for every
\begin{equation*}
	\phi \in C(0,T; \L^2)\cap L^2(0,T; \H_0^1)\cap H^1(0,T; \H^{-1})
\end{equation*}
it holds
\begin{multline*}
	\brak{q^\epsilon_t, \phi_t} 
	= \brak{q_0, \phi_0} + \int_0^t\brak{q^\epsilon_s, \partial_s\phi_s}ds + \int_0^t \brak{q^\epsilon_s, u^\epsilon_s\cdot \grad \phi_s }ds - \gamma\int_0^t\brak{q_s, \phi_s}ds \\ 
	- \eps^2\int_0^t\brak{\grad q^\epsilon_s - \grad W_s, \grad \phi_s}ds \\ 
	+ \sum_{k = 0}^\infty c_k\brak{\rho_k, \phi_t}W_t^k - \sum_{k = 0}^\infty \int_0^t c_k\brak{\rho_k, \partial_s\phi_s}W^k_sds.
\end{multline*}
		
\end{lemma}
Since we are considering additive noise, no explicit stochastic integral appears in the definition of a solution. This allows to adopt a pathwise approach by focusing on $\eta^\epsilon \coloneqq q^\epsilon - W$, which solves
\begin{equation}\label{eq:EtaEpsilon}
\begin{cases}
	\partial_t \eta^\epsilon + ({u}^\epsilon \cdot \grad) \eta^\epsilon + \gamma \eta^\epsilon 
	= \epsilon^2 \Delta \eta^\epsilon  - ({u}^\epsilon \cdot \grad) W  - \gamma W , \\
	{{u}^\epsilon = \grad^\bot (A+L)^{-1}( \eta^\epsilon + W)} , \\
	\eta^\epsilon_t|_{\partial D}=0.
\end{cases}
\end{equation}
The above is equivalent to the equation for $q^\eps$, as we understand it in the analogous weak sense.
Since $W_t$ is regular enough, it is easy to prove that if $q^\eps$ is a solution of \eqref{NSE} then $\frac{d}{dt}\eta^\eps$ belongs to the space $L^2(0, T; \H^{-1})$ almost surely. Its $L^{2}(0, T; \H^{-1})$-norm however diverges as $\epsilon\rightarrow 0$, unless we prove higher regularity of $\eta$ (cf. computations in \cref{lemma:ViscSolBound}). 
	
\begin{remark}
The viscous approximation we consider is non-physical: it is motivated by mathematical convenience and, as mentioned above, it avoids the confrontation with the (physical) creation of vorticity at the boundary $\partial D$.
Physical viscous approximations satisfying no-slip boundary conditions do not allow to prove inviscid limit results, either in the deterministic \cite{Kato} or the stochastic framework \cite{luongo21inviscid,butori2023large}.
\end{remark}
	

\subsection{Galerkin approximation} 

The following provides an a priori bound on the solution of \eqref{NSE} in $L^p$ norm, independent of $p$ and $\epsilon$. 
	
\begin{lemma} \label{lemma:ViscSolBound}
Let $q_0\in \L^\infty$ and assume $q^\epsilon_t$ to be a weak solution of \eqref{NSE}. Then $q^\epsilon\in L^\infty(0, T; \L^\infty)$ $\PP$-almost surely and, for all $k\in \N$,
\begin{equation*}
    \norm{q^\epsilon}{L_t^\infty\mathbf{L}_x^{2k}} \leq C_1\left(\gamma, T, \norm{q_0}{\mathbf{L}_x^\infty}, \norm{W(\omega)}{L_t^{\infty} \H_x^{5/2}}\right)
\end{equation*}
and 
\begin{equation}\label{est2}
	\norm{q^{\eps}}{L^2_t\H^1_x} \leq \frac{1}{\eps} C_2\left(T, \norm{q_0}{\mathbf{L}_x^\infty}, \norm{W(\omega)}{L_t^{\infty} \H_x^{5/2}}\right) ,
\end{equation}
where constants $C_1, C_2$ are in fact random variables as they depend on the noise sample $\omega\in\Omega$.
\end{lemma}
	
\begin{remark}\label{rmk:aboutlemma:ViscSolBound}
We will actually prove something more, namely that the bound still holds with the $\H^{2+\delta}$ norm of $W_t$ for each $\delta>0$. Since we will assume more regularity on the noise term in later arguments, we state a sub-optimal result.

Let us also remark that the addition of the term $-\eps^2\Delta W_tdt$ in \eqref{NSE} is essential for the latter estimates in the case $k\geq 1$. Without that additional term and the subsequent cancellation in the dynamics of $\eta^\epsilon$, the estimate would require better regularity of the noise $W$.
\end{remark}
	
\begin{proof}
We take a pathwise approach and assume that $\omega\in\Omega$ belongs to a set of full probability such that $W(\omega, t)\in L^{\infty}(0,T; \H^{5/2})$.
Let $\eta^\eps_t = q^\eps_t - W_t$ solve \eqref{eq:EtaEpsilon}. Since $\eta^\eps\in L^2(0,T; \H^1)$, it is possible to use $\eta^{\eps,i}|\eta^{\eps,i}|^{2k-2}$, $i=1,2, 3$ as test functions.
Considering the equation term by term,
\begin{multline*}
	\brak{\frac{d}{dt}\eta^{\epsilon,i}_t,\eta^{\epsilon,i}|\eta^{\epsilon,i}|^{2k-2}} 
	= \brak{|\eta^{\epsilon,i}_t|^{k-1}{\frac{d}{dt}\eta^{\epsilon,i}_t,|\eta^{\epsilon,i}_t|^{k} } } \\
	= \frac{1}{2k}\frac{d}{dt} \norm{\eta^{\epsilon,i}}{L^{2k}}^{2k} 
	=\norm{\eta^{\epsilon,i}}{L^{2k}}^{2k-1}\frac{d}{dt}\norm{\eta^{\epsilon,i}}{L^{2k}} ;
\end{multline*} 
also, by the incompressibility of $u^{\eps}$
\begin{equation*}
	\brak{\eta^{\epsilon,i}_t, u^{\epsilon,i}_t\cdot \grad[(\eta^{\epsilon,i}|\eta^{\epsilon,i}|^{2k-2})]}
	= \frac{2k-1}{2k}\brak{u^{\epsilon,i}_t, \grad[(\eta^{\epsilon,i}_t)^{2k}]} 
	= 0 ,
\end{equation*}
while the dissipative term yields
\begin{equation*}
	\brak{\grad\eta^{\epsilon,i}_t, \grad[(\eta^{\epsilon,i}|\eta^{\epsilon,i}|^{2k-2})]}
	= (2k-1)\norm{(\eta^{\epsilon,i}_t)^{k-1}\grad\eta^{\epsilon,i}_t}{L^{2}}^2 .
\end{equation*}
By H\"older inequality with $p=2k$ and $p'=\frac{2k}{2k-1}$ on the terms involving $W$ and $|\eta^{\epsilon,i}|^{2k-1}$, and combining the contributions of all layers,
\begin{multline}\label{eq:apriori1}
	\norm{\eta^{\epsilon}_t}{\L^{2k}}^{2k-1} \frac{d}{dt} \norm{\eta^{\epsilon}_t}{\L^{2k}} + \eps^2(2k-1)\norm{(\eta^{\epsilon}_t)^{k-1}\grad\eta^{\epsilon}_t}{\L^2}^2 \\ 
	\le 
	-\gamma \norm{\eta^{\epsilon}_t}{\L^{2k}}^{2k} + \norm{\eta^{\epsilon}_t}{\L^{2k}}^{2k-1}\Big( \norm{(u^{\epsilon}_t\cdot \grad) W_t}{\L^{2k}} + \gamma \norm{W_t}{\L^{2k}}\Big) .
\end{multline}
Neglecting positive terms on the left-hand side, the latter simplifies to 
\begin{equation*}
	\frac{d}{dt} \norm{\eta^{\epsilon}_t}{\L^{2k}} 
	\leq  
	-\gamma \norm{\eta^{\epsilon}_t}{\L^{2k}} + \norm{u_t^{\epsilon}}{\L^{2k}} \norm{\grad W_t}{\L^\infty} + \gamma \norm{W_t}{\L^{2k}} .
\end{equation*}
	In order to close the estimate, we only need to treat carefully the nonlinear term involving $\norm{u^{\epsilon}}{\L^{2k}}$, applying the regularizing property of the elliptic problem \ref{QGBVP} stated in \cref{elliptic estimate}. 
	Since $u^{\epsilon} = \op (\eta^{\epsilon} + W)$,
	there are two cases: using Sobolev embeddings one gets, for $k = 1,2$, 
\begin{equation*}
	\|u^{\epsilon}\|_{\L^k} \leq c\|u^{\epsilon}\|_{\H^1} \leq C( \|\eta^{\epsilon}\|_{\L^2}+ \|W_t\|_{\L^2}) ,
\end{equation*}
while for $k>2$
\begin{equation*}
	\|u^{\epsilon}\|_{\L^\infty} \leq c\|u^{\epsilon}\|_{\W^{1,4}} \leq C (\|\eta^{\epsilon}\|_{\L^4} + \|W_t\|_{\L^4}) ,
\end{equation*}
where the constants are independent of $\eps$ and $k$. 
	
We conclude by applying Gr\"onwall's Lemma. 
Consider first the case $k=1$: we deduce a uniform bound for the $\L^2$ norm of $\eta^\eps$ of the form 
\begin{equation*}
	\sup_{[0, T]}\norm{\eta^{\epsilon}_t}{\L^2} 
	\leq 
	C_1\left(\gamma, T, \norm{q_0}{\mathbf{L}_x^\infty}, \norm{W(\omega)}{L_t^{\infty} \H_x^{5/2}}\right). 
\end{equation*}
We then integrate \eqref{eq:apriori1} in time, getting rid of unnecessary terms, and obtain  
\begin{align*}
	{\epsilon}^2  \int_0^t \norm{\grad \eta^{\epsilon}}{\L^2}^2 
	\leq 
	\frac12 & \norm{\eta_0}{\L^2}^2 + \int_0^t\norm{\eta^{\epsilon}_s}{\L^2} \Big( \norm{(u_s^{\epsilon} \cdot \grad) W_s}{\L^2} + \gamma \norm{W_s}{\L^2} \Big) ds, 
\end{align*}
which implies the second statement.
Taking $k=2$, we use the bound we just proved to obtain a similar one for the $\L^4$ norm, again by Gr\"onwall's Lemma. 
Now, using the $\L^4$ estimate and knowing that $u^{\eps}$ only depends on $\norm{\eta^{{\epsilon}}}{L_t^\infty\L_x^4}$, we can estimate all other $2k$-norms uniformly in $k$. In conclusion we get, for all $k\ge 1$,
\begin{equation*}
	\sup_{[0, T]} \norm{\eta^{\epsilon}_t}{\L^{2k}} \leq C_1\left(\gamma, T, \norm{q_0}{\mathbf{L}_x^\infty}, \norm{W(\omega)}{L_t^{\infty} \H_x^{5/2}}\right) ,
\end{equation*}
with the constant $C_1$ independent of $k$.
This implies that $\eta^{\epsilon} \in L^\infty (0,T;\L^p) \cap L^2(0,T;\H^1)$ uniformly for all $p$, hence
\begin{equation*}
	\eta^{\epsilon} \in L^\infty (0,T;\L^\infty) \cap L^2(0,T;\H^1) .\qedhere
\end{equation*} 
\end{proof}
	
\begin{remark}\label{remark: a priori estimates}
The family of bounds in $L^\infty(0,T; \L^{2k})$ is stable under the weak$\star$ convergence of $\eta^{\eps}\rightharpoonup \eta$ in $L^\infty(0,T; \L^{2k})$. Indeed, the weak$\star$ convergence implies a uniform bound on the $L^\infty(0,T; \L^2)$ norms, which in turn gives a uniform bound in $L^\infty(0,T; \L^4)$, by what was showed in the lemma. As we have proved, this gives a uniform bound for all other norms,  meaning that, up to passing to a subsequence which converges weakly$\star$ in all $L^\infty(0,T; \L^{2k})$, the bounds hold for any limit function $\eta$ due to the lower-semicontinuity of the $\L^{2k}$-norms.
\end{remark}
	
\begin{proposition}\label{prop:existenceNSE}
Let $\epsilon>0$, $q_0\in \L^2$ and $W \in C([0,T];\H^{5/2})$ $\PP$-almost surely Then the problem \eqref{eq:EtaEpsilon} admits a unique solution $\eta^\epsilon \coloneqq q^\epsilon - W$. In particular, 
\begin{equation*}
	\eta^\epsilon \in C(0,T;\mathbf \L^2) \cap L^2(0,T; \H^1)\cap W^{1,2}(0, T; \H^{-1}) .
\end{equation*}
\end{proposition}
	
\begin{proof}
Consider a complete linearly independent system $\{e_k\}_{k\ge 0}$ in $\H^1(\Omega)$, made of eigenfunctions of $(A+L)$, and $\H_N = \operatorname{Span}(e_1, \dots,e_N)$.
Let $\eta^{\epsilon,N} (x,t) = \sum_{k = 1}^N c_k(t) \e_k(x)$, with coefficients  $c_k$ defined in such a way that $\eta^{\eps, N}$ solves the equation on $\H_N$, meaning that for every $k=1\ldots N$
\begin{multline*}
	\brak{ \partial_t \eta^{\eps, N}, e_k } 
	- \brak{ \eta^{\epsilon, N}, u^{\epsilon,N} \grad e_k }
	+ \gamma \brak{ \eta^{\epsilon,N}, e_k } \\
	=  
	- \epsilon^2 \brak{ \grad \eta^{\epsilon,N}, \grad e_k }
	- \brak{ (u^{\epsilon,N} \cdot \grad) W, e_k }
	- \gamma \brak{ W, e_k } .
\end{multline*}
Local well-posedness of the system above follows from classical results about stochastic differential equations with locally Lipshitz coefficients, cf. for example \cite{KaratzasSchreve,Skorokhod}. 
As for global well-posedness and the possibility to pass to the limit in $N$, we need some estimates on the solutions independent from $N$ and from time $t$. 
We give now an energy estimate: by testing $\eta^{\eps, N}$ against itself we obtain, arguing as in \cref{lemma:ViscSolBound}, 
\begin{multline*}
	\frac12 \frac{d}{dt} \norm{\eta^{\epsilon,N}_t}{\L^{2}}^2 + \eps^2 \norm{\grad\eta^{\epsilon,N}_t}{\L^2}^2 \\
	\leq 
	- \gamma \norm{\eta^{\epsilon,N}_t}{\L^{2}}^{2} 
	+ \Big( \norm{(u^{\epsilon,N}_t \cdot \grad) W_t}{\L^{2}} + \gamma \norm{W_t}{\L^{2}}\Big) \norm{\eta^{\epsilon,N}_t}{\L^{2}},
\end{multline*}
from which we obtain
\begin{multline}\label{energy est}
	\sup_{[0, T]}\norm{\eta^{\epsilon,N}_t}{\L^2} 
	\leq 
	C_1\left(\gamma, T, \norm{q^N_0}{\mathbf{L}_x^2}, \norm{W(\omega)}{L_t^{\infty} \H_x^{5/2}}\right) \\
	\leq 
	C_1\left(\gamma, T, \norm{q_0}{\mathbf{L}_x^2}, \norm{W(\omega)}{L_t^{\infty} \H_x^{5/2}}\right) .
\end{multline}
and consequently 
\begin{equation}\label{eq:est2b}
	\norm{q^{\eps,N}}{L^2_t\H^1_x} \leq \frac{1}{\eps} C_2\left(T, \norm{q_0}{\mathbf{L}_x^2}, \norm{W(\omega)}{L_t^{\infty} \H_x^{5/2}}\right).
\end{equation}
We can also say something about the regularity in time of $\eta^{\epsilon, N}$. 
Indeed, testing against a function $\phi \in \H^{1}$ yields
\begin{multline*}
	\brak{ \frac{d}{dt}\eta^{\epsilon,N} , \phi }
	+ \brak{ (u^{\epsilon, N} \cdot \grad) \eta^{\epsilon,N}, \phi } 
	+ \gamma \brak{ \eta^{\epsilon,N}, \phi } 
	= \\ 
	- \epsilon^2 \brak{ \grad \eta^{\epsilon,N}, \grad \phi }   
	- \brak{ (u^{\epsilon, N}\cdot \grad) W, \phi }
	- \gamma \brak{ W, \phi } .
\end{multline*}
By H\"older inequality, Sobolev embedding $H^1\hookrightarrow L^4$, \cref{elliptic estimate} we obtain
\begin{align*}
	&\int_0^T \Big| \Big\langle \frac{d}{dt} \eta^{\epsilon,N}_t , \phi \Big\rangle \Big|^2 dt\\ 
	&\qquad \leq 
	\int_0^T \left( \big\| u_t^{\epsilon, N} \big\|^2_{\L^2} \| \phi \|^2_{\H^1} \big\| \eta_t^{\epsilon,N} \big\|^2_{\L^4} + \gamma\big\| \eta_t^{\eps,N} \big\|_{\L^2} + \|\phi\|^2_{\H^1} \|\grad \eta_t^{\eps, N}\|_{\L^2} \right) dt \\ 
	& \qquad \quad
	+ \int_0^T \left( \|\phi\|^2_{\L^2} \|W_t\|^2_{\L^2} + \|u_t^{\epsilon, N}\|^2_{\L^2}\|\phi\|_{\H^1} \|W_t\|^2_{\L^4} \right) dt \\ 
	& \qquad \leq
	C\left( \gamma + \norm{W}{L^\infty_t\L_x^2} + \norm{\eta^{\epsilon,N}}{L^\infty_t\L_x^2} \right)^2 \left( \norm{W}{L^\infty_t\L_x^4} + \norm{\eta^{\epsilon,N}}{L^\infty_t\H_x^1} \right)^2 \norm{\phi}{H_x^1}^2 \\ \nonumber
	& \qquad \quad + \eps^2M^2 \norm{\phi}{H_x^1}^2,
\end{align*}
therefore, by \cref{energy est} and $\|\grad \eta^{\eps, N}\|_{L^2_t\L^2_x}\le M/\eps$ (following from \eqref{eq:est2b}),
\begin{equation}\label{eq: derivative bound}
	\int_0^T \Big| \Big\langle \frac{d}{dt} \eta^{\epsilon,N}_t , \phi \Big\rangle \Big|^2 dt
	\leq \frac1\epsilon C(\gamma, \|q_0\|_L^2, \omega)\|\phi\|_{H_x^1}^2.
\end{equation}
This shows that $\frac{d}{dt}\eta^{\epsilon,N}$ defines an operator in $ L^2(0,T;\H^{-1})$ whose norm is independent of $N$ (but diverges for $\eps\rightarrow 0$) and, together with \eqref{est2} and Aubin-Lions-Simon lemma (cf. \cite{:simon}), we obtain that the sequence $\{\eta^{\epsilon,N}\}_{N\geq0}$ is precompact in $C(0,T;\mathbf \L^2)$.
We can then extract a subsequence converging in $C(0,T;\mathbf \L^2)$ to some limit point $\eta^{\eps}$ satisfying \eqref{eq:EtaEpsilon} and assume weak convergence in $L^2(0, T; \H^1)$ and weak$\star$ convergence in $L^\infty(0, T; \L^2)$ . Considering the weak formulation for $\eta^{\eps, N}$,
\begin{multline*}
	\brak{ \eta_t^{\epsilon,N} , \phi } - \brak{ q_0^{N} , \phi } 
	+ \int_0^t \big[ \brak{ (u_s^{\epsilon, N} \cdot{} \nabla) \eta_s^{\epsilon,N}, \phi } 
	+ \gamma \brak{ \eta_s^{\epsilon,N}, \phi } \big] d s\\ 
	= -\int_0^t \big[  \epsilon^2 \brak{ \nabla \eta_s^{\epsilon,N}, \nabla \phi } 
	- \brak{ (u_s^{\epsilon, N}\cdot \nabla) W_s, \phi }
	- \gamma \brak{ W_s, \phi } \big] d s.
\end{multline*}
we are able to pass the equation to the limit $N\rightarrow +\infty$: weak convergence in $L^2(0,T; \H^1)$ yields convergence for the linear terms on the right hand side and convergence in $C(0,T, \L^2)$ assures the convergence of $\brak{ \eta_t^{\epsilon,N} , \phi }$, while the non-linear term converges due to the additional strong convergence of $u^{\eps, N}$ in $C(0, T; \H^{1})$.
We deduce that $\eta_t$ solves \eqref{NSE} and that
\begin{equation*}
	\eta^{\eps}\in C(0,T;\L^2)\cap L^{2}(0,T; \H^1) .
\end{equation*}
Finally, uniqueness follows by very similar arguments based on energy estimates that we omit. Let us notice that uniqueness implies the possibility of repeating existence proof without using subsequences. This gives for free the progressive measurability of $\eta$ as a process in $\L^2$, since any finite-dimensional pathwise approximation is measurable and converges (in the appropriate sense) to the limit, which is a continuous process in $\L^2$, without passing to $\omega$-dependent subsequences.
\end{proof}
	

\subsection{Vanishing viscosity}
	
Fix $q_0\in \L^\infty$.
The bound of $\frac{d}{dt}\eta^\eps$ in $L^2(0,T; \H^{-1})$ can be improved and made independent of $\epsilon$: it is sufficient to repeat the computations in \eqref{eq: derivative bound} estimating directly the $\L^4$ norm of $\eta^\eps$, thanks to \cref{lemma:ViscSolBound}, instead of using the Sobolev embedding $\H^1\hookrightarrow \L^4$. We are then ready to prove an existence result.
	
\begin{proposition}\label{prop:existence}
Let $\eta^\epsilon$ be a solution of \eqref{eq:EtaEpsilon}. 
For every sequence $\eps_n\rightarrow 0$ there exist a subsequence $\epsilon_{k_n}$ and a function $\eta$ such that $\eta^{\eps_{k_n}} \rightarrow \eta$ strongly in $C(0, T;\H^{-\delta})$ and weakly$\star$ in $L^\infty(0, T; \L^\infty)$. Moreover, $q_t = \eta_t + W_t$ is a solution of 
\begin{equation*}
\begin{cases}
	dq_t + ([U_t\cdot \grad ]q_t +\gamma q_t) dt = dW_t , \\
	U = \grad^\bot (A+L)^{-1}q , \\
	\grad \cdot U = 0
\end{cases}
\end{equation*}
in the sense of \cref{def:solutionqgsto} and $q_t \in C_w(0, T; \L^\infty).$
\end{proposition}
	
\begin{proof}
By \cref{lemma:ViscSolBound}, the family $\{\eta^\epsilon\}_{\epsilon\geq0}$ is bounded in $L^\infty(0,T; \L^2)$, uniformly in $\epsilon$, and since the bound in $L^2(0,T; \H^{-1})$ is now uniform in $\epsilon$ as well, by Aubin-Lions-Simon lemma $\{\eta^\epsilon\}_{\epsilon\geq0}$ is compactly embedded in  $C(0,T;\H^{-\delta})$. 
Therefore, for every sequence $\epsilon_n \rightarrow 0$ there exists a subsequence $\epsilon_{k_n}$ such that $\eta^{\epsilon_{k_n}}$ converges strongly to some $\eta$ in $C(0,T;\H^{-\delta})$, and weakly$\star$ in $L^\infty(0, T; \L^\infty)$. 
With these convergences, we are able to pass to the limit in the weak formulation of \eqref{eq:EtaEpsilon} exactly as in the proof of \cref{prop:existenceNSE} and obtain that $\eta$ solves 
\begin{multline*}
	\brak{ \eta_t , \phi }
	+ \int_0^t \big[ \brak{ (u_s \cdot \grad) \eta_s, \phi } 
	+ \gamma \brak{ \eta_s, \phi } \big] d s
	\\ =\brak{ q_0 , \phi } 
	+ \int_0^t  
	[\brak{ (u_s \cdot \grad) W_s, \phi }
	+ \gamma \brak{ W_s, \phi } \big] d s.
\end{multline*}
Now, since $\eta\in C(0, T; \H^{-\delta})\cap L^\infty(0, T; \L^\infty)$, by \cite[Ch. 3, Lemma 1.2]{temam} it holds that $\eta\in C_w(0, T; \L^\infty)$. Clearly, the same also holds for $q=\eta + W$ thanks to the regularity of $W$. 
\end{proof}
	
\begin{remark}
We have not proved measurability of the process $\omega \mapsto q_t(\omega)$. This will follow, as in the case of $q^\eps$, by the uniqueness result we prove in the following section, so we postpone the issue. However, by standard theory of Banach spaces the measurability of the process $q_t$ can be proved also without uniqueness of the solutions, arguing as in \cite[Section 3.1]{Bessaih1999}.
\end{remark}
	

\subsection{Uniqueness}

The following establishes the uniqueness statement in \cref{thm:main}.
	
\begin{proposition}\label{prop:uniqueness}
Given $q_0\in L^{\infty}$ and a stochastic basis $(\Omega, \F, \{\F_t\}_t, \PP, \{W_t\}_t)$, there is at most one pathwise solution $q$ to problem \eqref{eq:qgsto} in the space $L^{\infty}(0, T; \L^{\infty})$.
\end{proposition}
\noindent
In fact, we prove a stronger result than probabilistically strong well-posedness:
for $\PP$-a.e. $\omega\in\Omega$, existence and uniqueness hold for the deterministic equation with rough forcing $W(\omega)$. This notion is often referred to as ``path-by-path'' uniqueness. 

\begin{proof}
Assume there exist two solutions $q$ and $q'$ and let $u$, $u'$ be the associated velocities, $\psi$, $\psi'$ the stream functions. We denote their difference by  
\begin{equation*}
	\bar\psi = \psi - \psi' . 
\end{equation*}
By \cref{def:sol2}, for all test functions $\phi \in \mathbf{C}^1_{t,x}$ with compact support it holds
\begin{align*}
	\brak{ (A+L) \bar\psi_t, \phi(t)}  
	={} & \int_0^t\brak{(A+L)\bar\psi_s, \partial_s\phi(s)}ds  - \gamma\int_0^t\brak{(A+L)\bar\psi_s, \phi(s)}ds \\ 
	& + \int_0^t \brak{(A+L)\psi'_s, \grad^\perp \bar\psi_s\cdot \grad \phi(s) }ds \\ 
	& + \int_0^t \brak{(A+L)\bar\psi_s, u_s\cdot \grad \phi(s) }ds .
\end{align*}
If $q$ is a solution of \eqref{eq:qgsto} in $L^{\infty}(0, T; \L^{\infty})$, then $\frac{d}{dt}(q-W)\in L^2(0, T; \H^{-1})$, as it is easily derived like in \eqref{eq: derivative bound}. Therefore $ q-q'$ has the same regularity, and we can integrate by parts in time to obtain  
\begin{align*}
	\int_0^t\brak{ \partial_s(A+L)\bar\psi_s, \phi(s) } ds 
	={} & - \gamma\int_0^t\brak{(A +L)\bar\psi_s, \phi(s)}ds \\
	& + \int_0^t \brak{(A +L)\psi'_s, \grad^\perp \bar\psi_s\cdot \grad \phi(s) }ds \\ 
	& + \int_0^t \brak{(A +L)\bar\psi_s, u_s\cdot \grad \phi(s) }ds.
\end{align*}
By a density argument this is extended to $\phi \in L^{\infty}(0, T; \H^1)$, hence we can choose $\phi= \bar\psi$ and obtain 
\begin{multline*}
	\frac12\norm{{D^{1/2}}\grad \bar \psi_t}{\L^2}^2 
	= \frac12 \brak{ L\bar\psi_t, \bar \psi_t }  - \int_0^t \brak{(A +L)\psi'_s, \grad^\perp \bar\psi_s \cdot \grad \bar\psi_s }ds \\ 
	- \int_0^t \brak{(A +L)\bar\psi_s, u_s\cdot \grad \bar\psi_s  }ds + \gamma\int_0^t\brak{(A +L)\bar\psi_s, \bar\psi_s }ds\\
	=:I+II+III+IV.
\end{multline*}
Since we are assuming $L$ to be negative semi-definite the term $I$ is negative, and $II$ vanishes because 
$\grad^\perp \bar\psi_s\cdot \grad \bar\psi_s = 0 $. 
		
We now need estimates on the remaining terms. As for $III$, the idea is to integrate by parts as in \cite{YUDOVICH19631407}: the equality $u=\grad^\perp \psi$ leads, up to a constant depending on $A$, to 
\begin{equation*}
	\int_0^t\int_D \derp{^2\psi_s}{x_1\partial x_2}\Big[ \big(\derp{\bar\psi_s}{x_1}\big)^2 - \big(\derp{\bar\psi_s}{x_2}\big)^2\Big] + \Big( \derp{^2\psi_s}{x_2^2} - \derp{^2\psi_s}{x_1^2}\Big)\derp{\bar\psi_s}{x_1}\derp{\bar\psi_s}{x_2} dxds ;
\end{equation*}
taking absolute values we reduce ourselves to estimate the quantity 
\begin{equation}\label{yud1} 
	\int_0^t\int_D |\grad^2\psi_s||\grad \bar \psi_s|^2 dxds .
\end{equation}
By assumption $\psi \in L^\infty(0, T; B^\infty)$, where $B^\infty$ is the space of those $\psi$ solving \eqref{QGBVP} for some $q_0\in \L^\infty$ (it is a Banach space if endowed with the norm $\norm{\psi}{B^\infty}\coloneqq\norm{(A+L)\psi}{\L^\infty}$), so we do not have bounds on the $\L^\infty_{t,x}$ norm of the first factor, but we have a control on any $\L^p$ norm, for $p<\infty$: in fact, we know that $\grad\psi, \grad\psi'\in L^{\infty}_t(W^{1,p}_{0, x})^3\subset L^\infty_{t,x}$. 
Let $M^{\eps}\coloneqq\|\grad \bar \psi\|^{\eps}_{L^\infty_t\L^\infty_x}$, $\forall \epsilon>0$.
By H\"older inequality with conjugated exponents $2/(2-\eps)$ and $2/\eps$, and since by \cref{elliptic estimate}
\begin{equation*}
	\norm{\grad^2 \psi}{\L^{2/\eps}}\le \frac{C}{\eps}\norm{q}{\L^\infty} ,
\end{equation*}
\eqref{yud1} can be bounded by 
\begin{equation*}
	\frac{CM^\epsilon}
	{\eps}\norm{q}{L^\infty_t\L^\infty_x}\int_0^t\norm{\grad\bar\psi_s}{ \L_x^2}^{2-\eps}ds .
\end{equation*}
We are thus left with $IV$: we use again H\"older inequality with exponents $2$, $\infty$, and $2$, then we estimate $\norm{L\bar\psi_s}{\L^2} \leq \norm{L}{HS} \norm{\bar\psi_s}{\L^2} \leq C\norm{\grad \bar \psi_s}{\L^2}$, where $\norm{\cdot}{HS}$ is the Hilbert-Schmidt norm and the last inequality follows from the Poincaré inequality (since $\bar\psi \in \W^{1,2}_0$). Hence, it holds 
\begin{equation*}
	\int_0^t \brak{ L\bar\psi_s, u_s\cdot \grad \bar\psi_s }ds \leq C\int_0^t\norm{\grad \bar \psi_s}{\L^2}^2ds
\end{equation*}
Gathering all estimates together and setting $z_t\coloneqq|\grad\bar\psi_t|_{\L_x^2}$
we obtain
\begin{equation*}
	\frac12z^2_t \leq \frac{C^\epsilon_1}{\eps} \int_0^t z^{2-\eps}_sds + C_2\int_0^t z^2_sds .
\end{equation*}
Differentiation in time then yields
\begin{equation*}
	\frac{d}{dt}z_t \leq \frac{C^\epsilon_1}{\eps} z^{1-\eps}_t + C_2 z_t
\end{equation*}
from which, by Gr\"onwall inequality,
\begin{equation*}
	z_t \leq (C_1^{\eps}t)^{\frac{1}{\epsilon}}e^{tC_2} .
\end{equation*}
Recalling that $C_1^\epsilon$ is bounded in $\eps$, it is possible choose some $t_0\leq T$ such that  $C_1^{\epsilon}t<1$ whenever $t\leq t_0$. Eventually, sending $\eps$ to zero we obtain $z\equiv 0$ in $[0, t_0]$. We can then repeat the same argument starting from $t_0$ and see that $z\equiv 0$ on $[0,T]$. This, together with $\bar\psi = 0$ on $\partial D$, implies $\bar\psi\equiv 0$ and consequently $q\equiv q'.$
\end{proof}
	

\subsection{Stability Properties of Solutions}\label{sec:stability}

In order to complete the proof of \cref{thm:main} we need some further stability results, which lay the ground for the proof of the existence
of an invariant measure, \cref{thm:invmeasure}. We will denote by $q(t_0, t, \chi)$ the solution to \eqref{eq:qgsto} at time $t$, started at time $t_0$ with initial datum $\chi$. Whenever the initial time is $t_0=0$, it is omitted.
	
\begin{proposition}\label{prop: cont dep from data}
Let $\chi^n \weakstarto \chi$ in $\L^\infty(D)$ and let $q^n_t \coloneqq q(t_0, t, \chi^n)$ be the solution of \eqref{eq:qgsto} with initial datum $\chi^n$ at some fixed initial time $t_0$. 
Then $\PP$-almost surely $q_t^n \weakstarto q_t \coloneqq q(t_0, t, \chi) \in \L^\infty(D)$ for every $t>t_0$. 
\end{proposition}
	
\begin{proof}
By the a priori estimates in \cref{lemma:ViscSolBound} and by \cref{remark: a priori estimates}, $\eta^n_t= q^n_t- W_t$ is uniformly bounded in $L^{\infty}(0, T; \L^\infty)$ and $\frac{d}{dt}\eta^n_t $ is uniformly bounded in $L^2(0, T; \H^{-1})$, hence $\eta^n$ is relatively compact in $C(0, T; \H^{-\delta})$. 
Consequently we can extract a subsequence (still denoted as $\eta^n$) such that 
\begin{align*}
	\eta^n\rightarrow \eta \qquad & \in C(0, T; \H^{-\delta}) 
	\\
	\eta^n\weakstarto \eta \qquad & \in L^{\infty}(0, T; \L^\infty) 
\end{align*}
for some $\eta \in C(0, T; \H^{-\delta}) \cap L^{\infty}(0, T; \L^\infty)$.
The same holds for $q^n$, with limit $q = \eta + W_t$. 
Then, for any $g\in \L^1(D)$ 
\begin{equation*}
	\int_{t_0}^T  \brak{ q_t^n, g} dt\rightarrow \int_{t_0}^T \brak{ q_t, g} dt
\end{equation*}
and, testing against $\chi^\delta_t g$ for some $\chi^\delta_t$ localized around $t$ at scale $\delta$, we get, sending $\delta\rightarrow 0$,
\begin{equation}\label{weak convergence}
    \brak{ q_t^n, g} \rightarrow \brak{ q_t, g}
\end{equation}
$\PP$-almost surely for almost every $t\in [t_0, T]$. 
Now we are ready to see that $q_t$ solves \eqref{eq:qgsto}. In fact, arguing as in the proof of \cref{prop:existence} and using what we just showed for the first term, we are able to pass to the limit the weak formulation for $\eta^n$ for every $g\in \C_c^1(D)$, 
\begin{equation*}
	\brak{\eta^n_t, g} 
	= \brak{\chi^n, g} 
	+ \int_{t_0}^t \left( \brak{\eta^n_s, u^n\cdot\grad g} 
	- \gamma\brak{\eta^n_s,g} 
	+ \brak{W_s, (u^n_s) \cdot \grad g} 
	- \gamma\brak{W_s, \eta^n_s} \right) ds .
\end{equation*}
from which it follows that $q_t$ is the unique solution of \eqref{eq:qgsto} with initial data $\chi$, that $q_t\in C_w(0, T, \L^\infty)$ and consequently that \eqref{weak convergence} holds for all $t\in [0, T]$, but only if $g\in \C_c^1(D)$. We can conclude by considering $q_t^n-q_t$ as a linear operator on $\L^1$, for every $t>t_0$: it converges pointwise to zero on the dense subspace $\C^1_c(D)$ and, since $\{q_t^n\}$ and $q_t$ are uniformely bounded in $\L^\infty$, a simple density argument yields $\brak{q_t^n-q_t,g} \rightarrow 0$ for all $t\ge t_0$, $g\in \L^1$. 
\end{proof}
	
\cref{prop: cont dep from data} allows to derive a weak Feller property for the semigroup associated to the solution of our problem. Recall (cf. also \cref{sec:invariantmeasures}) that, for every $\phi\in B_b(\L^\infty, \tau_\star)$, such semigroup is defined as
\begin{equation}\label{eq:markov_semigroup}
	P_t\phi(\chi)\coloneqq\expt[]{\phi(q(t, \chi))}  
\end{equation}
	
\begin{corollary} \label{cor: feller}
The solution map $(t, \chi)\mapsto q(t_0, t, \chi)$ defines a sequentially weakly$\star$ Feller semigroup in $\L^\infty$, i.e. $P_t$ is continuous as an operator 
\begin{equation*}
	P_t: SC_b(\L^\infty, \tau_\star)\rightarrow SC_b(\L^\infty, \tau_\star)
\end{equation*}
on the space of bounded sequentially weakly$\star$ continuous functions ($f\in SC_b(\L^\infty, \tau_\star)$ if $f(\xi^n)\rightarrow f(\xi)$ whenever $\xi\weakstarto \xi$ in $\L^\infty$).
\end{corollary}
	
Since the weak$\star$ topology is in general non-metrizable, continuity and sequential continuity do not coincide. 
Considering instead the bounded weak$\star$ topology, continuity with respect to $\tau^b_{\star}$ is equivalent to sequential continuity with respect to $\tau_\star$, as recalled in \cref{sec:preliminaries}.
Hence, $P_t$ is continuous as an operator
\begin{equation*}
	P_t: C_b(\L^\infty, \tau^b_{\star})\rightarrow C_b(\L^\infty, \tau^b_{\star}) .
\end{equation*}
	
The forthcoming Proposition concerns regularity of solutions when the initial datum is in $\W^{1,4}$, a separable proper subspace of $\L^\infty$. Together with the dense embedding of $\W^{1,4}$ into $(\L^\infty, \tau_\star)$, this will be intstrumental in establishing the Markov property of the solution semigroup.
	
\begin{proposition}\label{prop: sol in w14}
Let $q_0\in\W^{1,4}$. Then the solution $q(t_0, t, q_0)$ to \eqref{eq:qgsto} started at $q_0$ belongs to $C_w(t_0, T; \W^{1, 4})$ $\PP$-almost surely and is adapted as a process in $\W^{1,4}$.
\end{proposition}
	
The proof consists in a Gr\"onwall argument and closely resembles that of \cite[Theorem 5]{Bessaih2020}. It is however possible to weaken the regularity hypothesis on $W_t$ required in the cited work: in order to highlight how so, we provide a sketch of the proof.
Let us state first a preliminary result, adapting a classic bound from elliptic PDE theory.
	
\begin{lemma}\label{lemma: kato}
Given a solution $q$ of \eqref{eq:qgsto}, if $u$ is the associated velocity it holds
\begin{equation*}
	\norm{\grad u}{\L^\infty}
	\leq 
	C \norm{q}{\L^\infty} \big(1 + \log_+ \big(\norm{\grad q}{\L^4_x}\big) \big) 
\end{equation*}
for some constant $C>0$.
\end{lemma}
	
\begin{proof}
The estimate is a consequence of the one contained in the work of A. Ferrari, \cite[Proposition 1]{Ferrari1993}. 
By definition, each component $u^i$ of $u$ solves the elliptic problem 
\begin{equation*}
	h_i\grad\times u^i = q^i - \sum_jL_{ij}\psi^j 
\end{equation*}
where both the coefficients $h_j$ of $D$ and the boundary conditions are as discussed in \cref{sec:overview}. Set for simplicity $Q^i\coloneqq q^i - \sum_j L_{ij}\psi^j$
By the estimate in \cite{Ferrari1993}, we then obtain
\begin{equation*}
	h_i\|\grad u^i\|_{\L^\infty} 
	\leq 
	\bar C \left( 1 + \log \left( 1 + \frac{\norm{Q^i}{\W_x^{1,4}}}{\norm{Q^i}{\L^\infty}} \right) \right) 
	\Big\|Q^i\Big\|_{\L^\infty} .
\end{equation*}
In fact, the regularity requirements in the statement of Ferrari are stricter, due to the fact that he works in a $3$-dimensional setting; since we only work in $2$ dimensions, reproducing his result requires weaker assumptions due to lower regularity loss in Sobolev embeddings. 
		
In order to recover the estimate we need, we can work a little more on the right-hand side. Using that $\log(x+y)\le \log 2 +\log_+x + \log_+y$, as well as $-x\log x\le 1/e$, we get to 
\begin{equation*}
	h_i \norm{\grad u^i}{\L^\infty} \leq C  \left( 1 + \norm{Q^i}{\L^\infty} \left( 1 + \log_+ \left( \norm{Q^i}{\W_x^{1,4}} \right) \right) \right) .
\end{equation*}
We then take the maximum over $i=1, 2, 3$, use the triangular inequality and again the properties of the logarithm again. This gives the estimate $\|L_{ij}\psi^j\|_{\W_x^{1,4}}\le C\|q^j\|_{\W_x^{1,4}}$ (same for the $\L^\infty$ norm, by \cref{elliptic estimate}). In the end, we get
\begin{equation*}
	\|\grad u\|_{\L^\infty}\le \bar{C}\|q\|_{\L^\infty}\left(C +\log_+\left(\|q\|_{\W^{1,4}_x}\right)\right) .
\end{equation*}
Finally, a few more computations and Poincaré inequality allow to substitute $\|\grad q\|_{\L^4}$ to $\|q\|_{\W^{1,4}_x}$, up to introducing new constants.
\end{proof}
	
\begin{proof}[Proof of \cref{prop: sol in w14} (sketch).]
By \cref{prop: cont dep from data}, we only need a bound on $\norm{\grad q}{\L^4}$. We begin by taking the gradient of \eqref{eq:qgsto}, then set $\eta = q - W$: for each layer $j$ and component $i$ we formally obtain the equation
\begin{equation*}
	\frac{d}{dt}\partial_{i}\eta^j + \partial_{i}(u^j\cdot \grad (\eta^j + W^j))+ \gamma\partial_i (\eta^j + W^j)=0 \qquad j=1,\ldots, 3 .
\end{equation*}
The following steps, as usual, should be performed first on the Galerkin's approximants of $\eta$ and then passed to the limit. We omit such passage for brevity. By multiplying the equation by $\partial_i\eta^j|\grad \eta^j|^2$, summing over $i$ and integrating over the spacial domain $D$, 
\begin{multline*}
	\frac14\frac{d}{dt}\norm{\grad\eta^j}{\L^2}^4 + \gamma\norm{\grad\eta^j}{\L^2}^4
	= - \gamma \sum_i^3 \brak{\partial_i W^j, \partial_i \eta^j|\grad\eta^j|^2}\\
	= - \sum_{i,k}^3\brak{\partial_i u^j_k\partial_k\eta^j, \partial_i\eta^j|\grad \eta^j|^2}
	- \sum_{i,k}^3\brak{\partial_i (u^j_k\partial_kW^j), \partial_i\eta^j|\grad\eta^j|^2}\\
	=:I+II+III.
\end{multline*}
The terms on the right-hand side are estimated by H\"older and Young inequalities: 
\begin{align*}
	|I| & \leq \frac{\gamma}{2}\norm{\grad \eta^j}{L^4}^4 +C_\gamma\norm{W^j}{W^{1,4}}, \\
	|II| & \leq C\norm{\grad u^j}{L^\infty}\norm{\grad\eta^j}{L^4}^4, \\
	|III| & \leq \sum_{i,k}^3\brak{\partial_i u^j_k \partial_kW^j + u^j_k\partial^2_{ik}W^j, \partial_i\eta^j|\grad\eta^j|^2} \\
	& \leq \norm{\grad \eta^j}{L^4}^4 + C\norm{\grad u^j}{L^4}^4 \norm{\grad W^j}{L^4}^4 + C'\norm{u^j}{L^\infty}^4 \norm{W^j}{W^{2,4}} .
\end{align*}
This way it becomes evident that the most regularity we need on the Brownian motion $W^j$ is for it to be in $W^{2,4}$. 
We can therefore limit ourselves, by Sobolev embedding, to requiring $W_t \in \H^{5/2}$. 
From the above computations it emerges that there still are three terms to estimate: $\norm{\grad u^j}{L^\infty}$, $\norm{\grad u^j}{L^4}$ and $\norm{u^j}{L^\infty}$. 
The first one is easily done via \cref{lemma: kato}. 
As for the second one, either \cref{lemma: kato} or \cref{elliptic estimate} and \cref{lemma:ViscSolBound} can be used. Finally, the third term is simply bounded by the second one. 
		
Thanks to the estimates on $I$, $II$ and $III$ we have that, summing over $j$,
\begin{multline*}
	\frac{1}{4} \frac{d}{d t}\norm{\grad \eta(t)}{\L^4}^4 
	\leq
	\norm{\grad \eta}{\L^4}^4 + C \norm{q}{\L^\infty} \left(C +\log_+\left(\|\grad q\|_{\L^4_x}\right)\right)  \norm{\grad \eta}{\L^4}^4 \\
	+ C \left( \gamma, t_0, T, \norm{q_0}{\L^\infty}, \norm{W}{C\left(0, T; \H^{5/2} \right)} \right) .
\end{multline*}
At this point, the proof goes along exactly as in \cite[Theorem 5]{Bessaih2020}, coming to 
\begin{equation*}
\sup_{t_0 \leq t \leq  T}\|\grad \eta \|_{\L^4} \leq C \left( \gamma, t_0, T, \norm{q_0}{\L^\infty}, \norm{W}{C\left(0, T; \H^{5/2} \right)} \right) .
\end{equation*}
	
Weak continuity in $\W^{1,4}$ is obtained as usual observing that it has been proved $\frac{d}{dt}\eta\in L^2(0, T; \H^{-1})$ and $\grad \eta \in L^{\infty}(0, T; \L^4)$\cite{temam}. Measurability is obtained by a standard argument, see for instance \cite{Bessaih1999}. 
\end{proof}
	
\begin{proof}[Proof of \cref{thm:main}]
We have already proved existence (\cref{prop:existence}) and uniqueness (\cref{prop:uniqueness}). 
Exploiting the additional integrability of $q$ and the density of $\H^1_0\cap \H^2$ into $\H^1_0$, we can relax the regularity requirements for test functions in the definition of weak solution and simply ask $\phi\in \H^1_0$, as 
\begin{equation*}
	\int_0^t\brak{q_s, U_s \cdot \grad \phi} d s
	\leq 
	\norm{q}{L^\infty_t\L^\infty_x} \norm{U}{L^\infty_t\L^2_x}\norm{\phi}{\H^1}
\end{equation*}
now gives a continuous operator on $\H^1_0$.
		
As for measurability, it follows by uniqueness and weak continuity in $\L^2$
.
Indeed, we know from \cref{prop:existenceNSE} that the Navier-Stokes approximants are measurable and, thanks to the well-posedness of the limiting problem, the whole sequence converges (in the appropriate sense) to the solution of the inviscid problem. 
Thus, since there is no dependence on $\omega\in\Omega$ in the choice of a suitable subsequence, solutions of the limit problem are also measurable.
	
Finally, stability in $\W^{1,4}$ is the content of \cref{prop: sol in w14} we just proved.
\end{proof}
	

\section{Invariant measure}\label{sec:invariantmeasures}
	

\subsection{Markov property}\label{ssec:markov}
	
The following technical result establishes in our context a well-known general principle, i.e. that time-homogeneous stochastic flows with independent increments are Markov processes. In order to be able to treat non-metric spaces, following \cite{Bessaih2020} we exploit a further assumption, that is the invariance of an embedded Banach space under the stochastic flow.
	
\begin{lemma}[Markov Property for Stochastic Flows]\label{lemma:markov}
Let $\pa{\Omega,\F,\F_t,\PP}$ be a filtered probability space and
let $(H,\mathcal{T})$ be a topological vector space.
		
Consider an $H$-valued stochastic flow $\xi_{s,t}^\chi(\omega)$,
i.e. a random field indexed by $0\leq s\leq t\leq T$, $\chi\in H$ such that:
\begin{itemize}
	\item $\xi_{s,s}$ is the identity map $\PP$-almost surely;
	\item $\PP$-almost surely, for all $s\leq r\leq t$, 
	$\xi_{s,t}^\chi=\xi_{r,t}^{\xi_{s,r}^\chi}$;
	\item if $\chi_n\to\chi$ in $H$ then $\PP$-almost surely for all $s,t\in [0,T]$ it holds $\xi^{\chi_n}_{s,t}\to\xi^\chi_{s,t}$ in $H$.
\end{itemize}
Assume that
\begin{itemize}
	\item (adaptedness) $\omega\mapsto \xi_{0,t}^\chi(\omega)$ is $\F_t$-measurable for all $t\in [0,T]$;
	\item (independence of increments) for all $s\leq t$, $\xi_{s,t}^\chi$ is independent of $\F_s$;
	\item (time homogeneity) for all $\chi\in H$, $\xi_{t,t+h}^\chi$ and $\xi_{s,s+h}^\chi$ have the same law for all $s,t\in [0,T]$ and admissible $h>0$;
	\item (continuous Banach embedding) there exists a separable Banach space $Y\subset H$ such that the embedding is dense and continuous. Moreover, for all $\chi\in Y$, $\PP$-almost surely $\xi_{s,t}^\chi\in Y$ for all $s, t \in [0,T]$.
\end{itemize}
Then the evolution operator
\begin{equation*}
	P_{t}:SC_b(H, \mathcal{T})\to SC_b(H, \mathcal{T}),\quad 
	(P_{t}\phi)(\chi) = \EE [\phi(\xi_{0,t}^{\chi})].
\end{equation*}
satisfies the Chapman-Kolmogorov equation: for all $\chi\in H$,
\begin{equation*}
    P_{t+s}\phi(\chi) = P_tP_s\phi (\chi).
\end{equation*}
\end{lemma}
	
\begin{proof}
It is enough to prove that 
\begin{equation*}
	\EE\bra{\phi\left(\xi_{t+s}^\chi\right)| \F_t} = P_s\phi\left(\xi_t^\chi\right),
\end{equation*}
as taking the expected value on both sides yields the thesis; in particular, the above equality is equivalent to 
\begin{equation*}
	\EE\bra{\phi\left(\xi_{t+s}^\chi\right)Z} = \EE\bra{P_s\phi\left(\xi_t^\chi\right)Z}
\end{equation*}
for every bounded $\F_t$-measurable random variable $Z$. Since $\xi_{t+s}^\chi = \xi_{t,t+s}^{\xi_{t}^\chi}$ almost surely, it is sufficient to prove 
\begin{equation}\label{eq:step1a}
	\EE\bra{\phi\left(\xi_{t,t+s}^\zeta\right)Z} = \EE\bra{P_s\phi\left(\zeta\right)Z}
\end{equation}
for every $H$-valued $\F_t$-measurable random variable $\zeta$.
		
Let us start by choosing $\zeta$ taking values in $Y$. 
Since $Y$ is a separable metric space, there exists a sequence  of $Y$-valued, $\F_t$-measurable random variables $\{\zeta_n\}_n \overset{Y}{\rightarrow} \zeta$ of the form 
\begin{equation*}
	\zeta_n = \sum_{i=1}^{k_n} \zeta_n^{(i)} \mathbbm{1}_{A_n^{(i)}} ,
\end{equation*}
with $\zeta_n^{(i)}\in Y$ and $\{A_n^{(1)}, \dots, A_n^{(k_n)}\}\subset \F_t$ a partition of $\Omega$. 
Convergence in $Y$ implies convergence in $(H,\mathcal{T})$ and, due to the continuity of $\xi$ with respect to initial conditions and to the sequential continuity of $\phi$ and $P_s\phi$, it is enough to prove 
\begin{equation}\label{eq:step1b}
	\EE\bra{\phi\left(\xi_{t,t+s}^{\zeta_n}\right)Z} = \EE\bra{P_s\phi\left(\zeta_n\right)Z} \quad \forall n .
\end{equation}
\cref{eq:step1a} then follows by dominated convergence (whenever $\zeta(\omega)\in Y$ $\PP$-almost surely).
		
Given that 
\begin{equation*}
	P_s\phi \left(\zeta_n\right) = \sum_{i=1}^{k_n} P_s\phi\left(\zeta_n^{(i)}\right) \mathbbm{1}_{A_n^{(i)}} \quad \PP-a.s.
\end{equation*}
and that 
\begin{equation*}
	\phi \left(\xi_{t,t+s}^{\zeta_n}\right) = \sum_{i=1}^{k_n} \phi\left(\xi_{t,t+s}^{\zeta_n^{(i)}}\right) \mathbbm{1}_{A_n^{(i)}}
\end{equation*}
(because $\xi_{t,t+s}^{\zeta_n} = \sum_{i=1}^{k_n} \xi_{t,t+s}^{\zeta_n^{(i)}} \mathbbm{1}_{A_n^{(i)}}$), we can reduce to proving that 
\begin{equation*}
	\EE \bra{ \phi\left(\xi_{t,t+s}^{\eta}\right) \mathbbm{1}_{A} Z } 
	= 
	\EE	\bra{ P_s\phi\left(\eta\right) \mathbbm{1}_{A} Z } 
\end{equation*}
for all $Y$-valued, $\F_t$-measurable random variables $\eta$  and for all $A\in\F_t$. Moreover, we can neglect the term $\mathbbm{1}_{A}$, as it is itself a bounded $\F_t$-measurable random variable.
In sight of this, we have
\begin{multline*}
	\EE \bra{ \phi\left(\xi_{t,t+s}^{\eta}\right) Z }
	=\EE \bra{ \EE \bra{ \phi\left(\xi_{t,t+s}^{\eta}\right) Z | \F_t } }
	=\EE \bra{ \EE \bra{ \phi\left(\xi_{t,t+s}^{\eta}\right) } Z } \\
	=\EE \bra{ \phi\left(\xi_{t,t+s}^{\eta}\right) } \EE \bra{ Z } 
	=\EE \bra{ \phi\left(\xi_{s}^{\eta}\right) x } \EE \bra{ Z }
	=P_s\phi\left(\eta\right) \EE \bra{ Z } 
	=\EE	\bra{ P_s\phi\left(\eta\right) Z },
\end{multline*}
which proves \eqref{eq:step1b}.
		
We are thus left to verify that \eqref{eq:step1a} holds for every $H$-valued, $\F_t$-measurable random variable we exploit the dense embedding of $Y$ into $H$. Let $\chi(\omega)\in H$ $\PP$-almost surely and $\{\chi_n\}_{n\in\N}$ be a sequence of $Y$-valued random variables converging to $\chi$ with respect to $\mathcal{T}$. We have already proved that for all $n$ it holds
\begin{equation*}
	\EE\bra{\phi\left(\xi_{t+s}^{\chi_n}\right)Z} = \EE\bra{P_s\phi\left(\xi_t^{\chi_n}\right)Z}
\end{equation*}
whenever $\phi\in SC_b(H, \mathcal{T})$ and $Z$ is a bounded, $\F_t$-measurable random variable. By continuity of $\xi$ with respect to $\chi_n$ we have that both $\phi\left(\xi_{t+s}^{\chi_n}\right) \rightarrow \phi\left(\xi_{t+s}^{\chi}\right)$ and $P_s\phi\left(\xi_t^{\chi_n}\right) \rightarrow P_s\phi\left(\xi_t^{\chi}\right)$, hence by dominated convergence their expected values converge as well, yielding \eqref{eq:step1a}.
\end{proof}
	
\begin{corollary}
Let $P_t$ be the solution semigroup for the problem \eqref{eq:qgsto}, as defined in \eqref{eq:markov_semigroup}. Then $P_t$ enjoys the Markov property, i.e. for any $s,t\geq 0$, $\phi \in SC_b(\L^\infty, \tau_\star)$ 
\begin{equation*}
	P_{t+s} \phi = P_t P_s \phi .
\end{equation*}   
\end{corollary}

\begin{proof}
It is easy to verify that $P_t$ meets the conditions of \cref{lemma:markov}. In particular, properties of the stochastic flow can be derived from \cref{thm:main} and from the properties of the Wiener process we chose as stochastic forcing. As for the continuous Banach embedding, we can use the dense embedding of the Banach space $\W^{1,4}$ into $(\L^\infty, \tau_\star)$.
\end{proof}

 
\subsection{Construction of the invariant measure}
	
We are finally able to show the existence of an invariant measure for our problem. 
	
\begin{proof}[Proof of \cref{thm:invmeasure}]
Let $q$ be the pseudo-vorticity solution to \eqref{eq:qgsto} as in \cref{thm:main}, started at $t_0=0$, and let
\begin{equation*}
	m_t\coloneqq \operatorname{\mathcal{L}}(q_t)
\end{equation*}
be its law on $\B(\tau^b_{\star})$. Consider the averaged measures
\begin{equation*}
	\mu_n\coloneqq\frac{1}{n}\int_0^n m_tdt.
\end{equation*}
This integration is allowed since we have proved joint measurability of the process $q_t$, with respect to the corresponding topologies, in \cref{thm:main}. The family $\{\mu_n\}_{n\in \N}$ is tight (see next section for the proof). 
On a separable metric space one would thus use Prokhorov's theorem to extract a subsequence converging in the sense of measures. 
However, the space $(\L^\infty, \B(\tau_\star^b))$ being non metrizable, we have to resort to the variant given in \cite[Theorem 3]{jakubowski98}. 
What we have to verify, in order to apply such variant, is the existence of a countable family of functionals $\{f_i\}_{i\in\N}$ on $\L^\infty$ taking values in $[-1,1]$, $\tau^b_{\star}$-continuous and point-separating in $\L^\infty$. 
The condition is satisfied: since all elements of the separable space $\L^1$ are $\tau^b_{\star}$-continuous (they are $\tau_\star$-continuous), any countable dense subset of $\L^1$ has the required property.
Hence, there exists a subsequence $\mu_{n_k}$ converging to some probability measure $\mu$ in the following sense:
\begin{equation*}
	\int f d\mu_{n_k}\rightarrow \int f d\mu \qquad \forall f\in C_b((\L^\infty, \tau^b_{\star}); \R) .
\end{equation*}
		
It is then immediate to conclude that $\mu$ is an invariant measure, since for every $\phi \in C_b((\L^\infty, \tau^b_{\star}); \R)$
\begin{equation*}
	\brak{P_t\phi, \mu_{n_k}} 
	= 
	\frac{1}{n_k} \int_{n_k}^{t+n_k} \brak{m_r, \phi} d r + \brak{\phi, \mu} - \frac{1}{n_k} \int_0^t \brak{m_r, \phi} d r 
\end{equation*}
and the right-hand side converges to $\brak{\phi, \mu}$. 
By \cref{prop: cont dep from data} $P_t\phi\in C_b((\L^\infty, \tau^b_{\star}); \R)$, thus it also holds that $\brak{P_t\phi, \mu_{n_k}}\rightarrow\brak{P_t\phi, \mu}$ and by uniqueness of the limit 
\begin{equation*}
	\brak{P_t\phi, \mu} = \brak{\phi, \mu} .\qedhere
\end{equation*}
\end{proof}
	

\subsection{Tightness of averaged measures}
	
\begin{proposition}\label{prop:flandolitrick}
Let $q(t_0, t, \chi)$ be solution of \eqref{eq:qgsto}. There exists a real-valued, $\PP$-almost surely finite random variable $r$ such that 
\begin{equation*}
	\sup_{t_0 \leq 0} \norm{q(t_0, 0, q_{t_0}=0)}{\L^\infty} \leq 
	r 
    \qquad \PP\text{-a.s}.
\end{equation*}
\end{proposition}
	
\begin{proof}
Consider, for $\lambda >0$, the stochastic convolution 
\begin{equation*}
	\zlam(t) \coloneqq \int_{-\infty}^t e^{-\lambda(t-s)} d W_s ,
\end{equation*}
where $\zlam$ has one component per layer.
It is known that $\zlam$ is the unique solution of the linear equation $d\zlam(t)=-\lambda\zlam(t) + dW(t)$ and has the same regularity as $W_t$. Set $\theta_\lambda(t) \coloneqq q(t_0, t, 0) - \zlam(t)$. $\theta_\lambda$ solves the deterministic equation 
\begin{equation*}
	\frac{d}{dt}\theta_\lambda= - \BS(\tlam + \zlam) \cdot \grad \tlam - \BS(\tlam + \zlam) \cdot \grad \zlam - \gamma \tlam + (\lambda - \gamma) \zlam
\end{equation*}
where $\BS$ formally stands for the operator $\grad^\perp(A+L)^{-1}$, in analogy with the Biot-Savart kernel usually found in the formulation of the standard Euler equation. 
		
Proceeding as in \cref{sec:wellposedness}, i.e. by employing a vanishing viscosity approximation (cf. also \cite{Bessaih2020}) and with computations analogous to those used to prove \cref{lemma:ViscSolBound}, we can test $\tlam$ against $|\tlam|^{p-2}\tlam$ and obtain 
\begin{equation*}
	\frac{d}{dt} \norm{\tlam}{\L^p} 
	\leq 
	C  \norm{\grad \zlam}{\L^\infty} \norm{\zlam}{\L^p}
	+ |\lambda - \gamma| \norm{\zlam}{\L^p}
	+ \left( C \norm{\grad \zlam}{\L^\infty} - \gamma \right) \norm{\tlam}{\L^p},
\end{equation*}
which yields, by Gr\"onwall inequality on $[t_0, 0]$, 
\begin{multline*}
	\norm{\tlam(0)}{\L^p} 
	\leq 
	\norm{\tlam(t_0)}{\L^p} e^{\int_{t_0}^0 (C \norm{\grad \zlam(s)}{\L^\infty} - \gamma) d s} \\ 
	+ \int_{t_0}^0 \left( C \norm{\grad \zlam(r)}{\L^\infty} + |\lambda - \gamma| \right) \norm{\zlam(s)}{\L^p} e^{\int_{s}^0 (C \norm{\grad \zlam(r)}{\L^\infty} - \gamma) d r} d s .
\end{multline*}
Taking the limit $p\rightarrow\infty$ we obtain, since $H^{\alpha-1} \subset L^{\infty}$ for $\alpha > 2$ and $\tlam(t_0) = \zlam(t_0)$, 
\begin{multline}\label{eq: flandoli trick}
	\norm{\tlam(0)}{\L^\infty} 
	\leq \norm{\zlam(t_0)}{\H^\alpha} e^{\int_{t_0}^0 (C \norm{\zlam(s)}{\H^\alpha} - \gamma) d s} \\ 
	+ \int_{t_0}^0 C \left( \norm{\zlam(s)}{\H^\alpha} + |\lambda - \gamma| \right) \norm{\zlam(s)}{\H^\alpha} e^{\int_{s}^0 (C \norm{\zlam(r)}{\H^\alpha} - \gamma) d r} d s ,
\end{multline}
where we allowed the constant $C$ to change due to Sobolev embedding.
		
A uniform bound on the right-hand side of \eqref{eq: flandoli trick} can now be obtained working term-by-term as follows. 
First, notice that 
\begin{equation*}
	\EE\bra{\norm{\zlam(t)}{\H^\alpha}^2} = \frac{1}{2\lambda} \EE\bra{\norm{W(1)}{\H^\alpha}^2} \qquad \forall t\in \R .
\end{equation*}
By ergodicity of $\zlam$ (cf. \cite{DaPrato2014}),
\begin{equation*}
	\lim_{t_0\rightarrow -\infty} \frac{1}{-t_0} \int_{t_0}^0 \norm{\zlam(s)}{\H^\alpha} d s = \EE \bra{ \norm{\zlam(0)}{\H^\alpha} } \qquad \PP\text{-a.s.}
\end{equation*}
and we can choose $\lambda$ in such a way that 
\begin{equation*}
	\lim_{t_0\rightarrow -\infty} \frac{1}{-t_0} \int_{t_0}^0 \norm{\zlam(s)}{\H^\alpha} d s < \frac{\gamma}{2C} \qquad \PP\text{-a.s.},
\end{equation*}
$C$ coming from \eqref{eq: flandoli trick}.
This means that there exists a random time $\tau(\omega) \leq 0$ $\PP$-almost surely such that, for every $t_0 < \tau(\omega)$, 
\begin{equation*}
	\int_{t_0}^0 C \norm{\zlam(s)}{\H^\alpha} d s \leq \frac12\gamma (-t_0) \leq 0.
\end{equation*}
Hence, for $- t_0$ and $\lambda$ large enough, $\exp\left(\int_{t_0}^0 (C\|\zlam(s)\|_{H^\alpha} - \gamma) ds\right) \leq \exp\left( \frac{\gamma}{2} t_0 \right)$. 
On the other hand, by continuity of $\zlam(t)$ in $\H^{\alpha}$ there exists a $\PP$-almost surely finite random variable $r_1$ such that
\begin{equation*}
	\sup_{\tau(\omega)\le t_0 \le 0}\int_{t_0}^0\|\zlam(s, \omega)\|_{\H^{\alpha}}ds \le r_1(\omega) \qquad \PP\text{-a.e. } \omega \in \Omega ,
\end{equation*}
and this yields a uniform bound for the exponential term, which moreover vanishes as $t_0\rightarrow-\infty$. 
In a similar way one can also obtain, for $t<0$, a uniform bound on $\norm{\zlam}{\H^\alpha}$ of the form $\norm{\zlam(t,\omega)}{\H^\alpha} \leq r_2(1 + |t|)$ for some $\PP$-almost surely finite random variable $r_2$.
We have thus proved that
\begin{equation*}
	\sup_{t_0 \leq 0} \norm{\theta_\lambda\left(0, \tlam(t_0)=-\zlam(t_0)\right)}{\L^\infty} 
	\leq 
	r_3(\omega)
\end{equation*}
for some $\PP$-almost surely finite random variable $r_3$. The thesis now follows recalling that $q = \tlam + \zlam$.
\end{proof}
	
Tightness of averaged measures is a direct corollary of the above proposition.
	
\begin{corollary}
For every $\eps>0$ there exist a real number $R_\eps>0$ such that 
\begin{equation*}
	\inf_{t\ge 0} \PP \left( \norm{q(0, t, q_0=0)}{\L^\infty} < R_\eps \right) > 1 - \eps .
\end{equation*}
In particular, the family of measures $\{\mu_n\}_{n\in\N}$ is tight in $\left(\L^{\infty}, \tau^b_{\star}\right).$
\end{corollary}
	
\begin{proof}
By time homogeneity $q(t_0, 0, q_{t_0}=0) \sim q(0, -t_0, q_0 = 0)$ $\forall t_0<0$. Since for all random variables $r$ $\PP$-almost surely finite and $\epsilon>0$ there exists $R_\epsilon>0$ such that $\PP(r < R_\epsilon) > 1-\epsilon$,
by \cref{prop:flandolitrick} we have the following uniform estimate:
\begin{multline*}
	\PP \left( \norm{q(0, t, q_0=0)}{\L^\infty} < R_\eps \right) 
	= \PP \left( \norm{q(-t, 0, q_{-t}=0)}{\L^\infty} < R_\eps \right)\\
	= \PP(r < R_\epsilon) > 1 - \eps .\qedhere
\end{multline*}
\end{proof}
	
\begin{acknowledgements}
F. G. and L. R. were both supported by the project \emph{Mathematical methods for climate science}, funded by he Ministry of University and Research (MUR) as part of the PON 2014-2020 "Research and Innovation" resources - Green Action - DM MUR 1061/2022
\end{acknowledgements}
	
\bibliography{main}{}
\bibliographystyle{plain}
	
\end{document}